\documentclass[letterpaper,10pt]{article}
\usepackage{amsmath}
\usepackage{amsthm}
\usepackage{amssymb}
\usepackage{amsfonts}
\usepackage[mathscr]{eucal}
\usepackage{amsxtra}     
\usepackage{authblk}
\usepackage[toc,page]{appendix}

\newcommand{\R}{\mathbb{R}}
\newcommand{\e}{\epsilon}

\newcommand{\vphi}{\varphi}
\newcommand{\supp}{\mbox{supp}\,}
\newcommand{\sgn}{\mbox{sgn}}

\newcommand{\T}{\mathcal{T}}
\newcommand{\Say}{\mathcal{S}_a^y}
\newcommand{\half}[1]{{#1}^{\frac{1}{2}}}
\newcommand{\mhalf}[1]{{#1}^{-\frac{1}{2}}}

\theoremstyle{plain}
\newtheorem{theorem}{Theorem}
\newtheorem*{theorem*}{Theorem}
\newtheorem*{thm}{Theorem \ref{sharp}}
\newtheorem{ppn}[theorem]{Proposition}
\newtheorem{corollary}[theorem]{Corollary}
\newtheorem{lem}[theorem]{Lemma}

\newtheorem{clm}[theorem]{Claim}

\theoremstyle{definition}

\newtheorem{case}{Case}[theorem]

\theoremstyle{remark}
\newtheorem{rem}{Remark}

\numberwithin{theorem}{section}
\numberwithin{equation}{section}
\title{Nonstandard estimates for a class of 1D dispersive equations and applications to linearized water waves}
\author{Jennifer Beichman%
\thanks{This work was partially supported by the Institute for Mathematics and its Applications and DMS 1101434}}
\affil{University of Wisconsin-Madison} 
\date{}
\begin{document}
\maketitle
\begin{abstract}
In this work, we obtain decay bounds for a class of 1D dispersive equations that includes the 
linearized water wave. These decay bounds display a surprising growth factor, which we show is sharp. The proofs rely on careful analysis of certain oscillatory integral operators.  In addition, these results 
have applications to the linearized water wave operator for low regularity data. 
\end{abstract}


\section{Introduction}

In this paper, we will explore the decay rate of a certain class of dispersive operators. The primary motivation is the work of  Wu \cite{Wu2009} on the almost global existence of solutions for the water wave equation. 
The key ingredients in \cite{Wu2009}  are an estimate on the dispersive decay rate of the linear water wave operator \begin{equation}\label{lwwWu} \partial^2_t -i\partial_\alpha\end{equation}
 and a change of both independent and dependent variables to transform out the quadratic nonlinearity in the full nonlinear water wave equation. The result in \cite{Wu2009}  roughly says that a 2D water wave with small initial height, energy, and slope in Sobolev space remains small and smooth almost globally in time. Intuitively, we expect for long time existence that we will only need the slope of the initial wave to be small, not the height and energy. In fact, Wu proved global existence for the 3D version of this problem with a smallness assumption only on the slope of the initial wave (cf. \cite{Wu2011}).  

The long time existence of the full water wave problem in two dimensions is a subject of much recent research. The works of Ionescu and Pusateri \cite{Ionescu2014} and Alazard and Delort \cite{ADelort2013, ADelort2013a} extend the result of \cite{Wu2009} to global existence by understanding further the nature of the cubic nonlinearity in the water wave equation. See 
  \cite{Ifrim2014} for a different proof.  All of these works assume data of similar types to that in \cite{Wu2009}; in particular, the issue of whether the assumption on the height and energy can be removed is not addressed. In this paper, we will examine the small height assumption and, in the process, prove some new linear decay estimates.  

 In the arguments of \cite{Wu2009}, the small height requirement comes from the control on the linear dispersive effects. These effects are controlled using a variant on Klainerman's method of invariant vector fields in this earlier work. In the case of the operator \eqref{lwwWu},  the invariant vector fields used are \[\Gamma = \left\{\partial_t, \partial_\alpha,L= \frac{t}{2}\partial_t+\alpha\partial_\alpha,\Omega_0=\alpha\partial_t +\frac{it}{2}\right\}.\] Using this set of vector fields, Wu derived a decay estimate of the $L^\infty$ norm of the solution in terms of the generalized $L^2$  Sobolev norms of the solution as well as the antiderivative of the solution. Notice that the final element here is, in fact, not a vector field in the traditional sense as it fails to satisfy a product rule. It is the use of this differential operator that leads to the reliance on the antiderivative, the source of the small height requirement.  To determine the necessity of the smallness assumption on the initial height, 
we need to understand the decay rates  for the solutions of the linearized water wave equation without relying on the control of the $L^2$ norm of the antiderivative of the solution at the initial time. 

From a physical perspective,  removing the dependence on the $L^2$ norm of the antiderivative of the solution at the initial time amounts to allowing more small frequency waves initially.\footnote{Notice that $\int f \in L^2$ if and only if $|\xi|^{-1}{\hat f}(\xi)\in L^2$, and thus limits the amount of small frequency present in the initial data.} Small frequency waves are a common presence in the ocean, it is important from scientific point of view the role of small frequency waves and its effects on large time behaviors of the solutions.

In this paper, we study the effect of various amounts of small frequency waves at the initial time to the decay rate of the solutions, and we show that by controlling only the $\dot H^s$, $s>0$ norm of the initial data \footnote{ Notice that for $f\in \dot H^s$, the Fourier transform $|\xi|^s\hat f(\xi)\in L^2$, which allows  the initial data to have $\hat u^(\xi, 0)\approx |\xi|^{-s-1/2}$  at $|\xi|\sim 0$.} the solution exhibits a surprising growth factor in time; furthermore, we show that this result is sharp. This factor identifies one of the possible growth mechanisms for solutions of the water wave equation. In addition, by balancing the study on the $\dot H^s\to L^\infty$ bound and the good decay provided by the estimate in \cite{Wu2009} for wave frequencies bounded away from the origin, we improve on both sets of decay estimates. Finally,  we give a precise relationship between singularities at the origin in the initial data and the rate of spatial decay of the solution to the linear water wave 
problem. The combination of these two results give us the precise threshold on the initial data for growth or decay in linearized estimates. 

In our proofs, we will not use the troublesome vector field $\Omega_0$. Our techniques are inspired by the work of Keel, Smith, and Sogge \cite{Keel2004}, where the authors studied the nonlinear wave equation in a domain with an obstacle by using a reduced set of vector fields. 

This is paper is organized as follows: in section 2, we will outline some preliminary results. In section 3, we will prove a set of dispersive estimates along integral curves of the scalar vetor field, which have a sharp growth factor coming from a certain collection of frequencies. These estimates naturally generalize to a  class of linear dispersive operators, so we present the results for the entire class of operator. Finally, in section 4, we will prove that for solutions to linearized water wave problem, a singularity in the initial data decreases the decay at spatial infinity. In this part, we only consider the linearized water wave operator for the sake of clarity. We begin this paper by laying the groundwork necessary to prove the new decay results in their full generality.  

\section{Preliminaries}
We consider the following general class of one dimensional dispersive differential equations. For $0<a$, $a\neq 1$, let $u(t,x)$ be a solution to the initial value problem

\begin{equation}\label{halfde}
 \left\{\begin{array}{l}
\partial_t u -i|D|^{a}u =0\\
u(0,x) = u_0(x).\\
\end{array}
\right.
\end{equation}
We define $|D|$ by the following Fourier transform:  \[|D| f := \int e^{-ix\xi} |\xi| \widehat{f}d\xi.\] Observe that $a=2$ gives the linear Schrodinger operator and $a=1/2$ is approximately one direction of propagation for the linearized two dimensional water wave equation. 

Following Klainerman, the set of vector fields we are interested in for equation \eqref{halfde} is $\{\partial_t, \partial_x, L = t\partial_t+(x/a)\partial_x , \Omega = x\partial_t +\frac{t}{2}\partial_x|D|^{-1} \}$,
with $L$ the scaling vector field and $\Omega$ the analog of $\Omega_0$ in \cite{Wu2009}. Inspired by the work of Keel, Smith and Sogge \cite{Keel,Keel2002,Keel2004}, our first step is to find a Sobolev type bound 
using the subset  $\Gamma = \{\partial_t, \partial_x, L = t\partial_t+(x/a)\partial_x\} $:

\begin{lem}\label{Sob}
For any $C^1(\R^+, \R)$ function $v(t,x)$ such that $v$ decays to zero as $|x|\rightarrow\infty$ and any parameter $y\in\R$, we have
\begin{equation}\label{Sobeqn}
\sup_{T\leq t\leq 2T}|v(t, yt^{1/a})|\leq \frac{C}{T^{1/2}}\sum_{k= 0}^1\left(\int_T^{2T}\left|L^kv (t, y t^{1/a})\right|^2dt\right)^{\frac{1}{2}}\end{equation}
\end{lem}

The lemma follows from a straightforward application of an averaged fundamental theorem of calculus. We can then control the absolute value of a solution by 
by controlling solutions restricted to integral curves in $L^2$. The right hand side of this estimate is the focus of the bulk of the work. 

We begin by rewriting $u(t,yt^{1/a})$ as an operator on the initial data using Fourier transforms. Define the operator $\mathcal{S}_a^y$ as \[\mathcal{S}_a^y v(t) = \int e^{i(yt^{1/a}\xi +t|\xi|^a)}\widehat{v}(\xi)d\xi.\] Therefore, $u(t,yt^{1/a})= \Say u_0$. In practice, we will suppress the sub and superscript. We can rewrite the $L^2$ norm with respect to this new operator: 
\begin{align*}
\half{\left(\int_T^{2T} |u(t,yt^{1/a})|^2dt\right)} &\leq \sup_{g\in L^2, \|g\|=1}|\langle\mathcal{S} u_0(t)\varphi_T(t), g(t) \rangle|\\
&= \sup_{g\in L^2, \|g\|=1}|\langle \widehat{u_0}(\xi), \widehat{\mathcal{S}^*}(\varphi_T g)(\xi) \rangle|
\end{align*} where $\varphi(t)\in C^\infty_0(\R)$ with $\varphi =1$ for $t\in(1,2)$ and $\varphi = 0$ for $t\in (1/2, 5/2)^C$ and $\varphi_T(t) = \varphi(t/T)$. For simplicity, we will let $\T = \widehat{\mathcal{S}^*}$, so \[\T h(\xi) =\int e^{-i(yt^{1/a}\xi+t|\xi|^a)}h(t) dt.\] The proof of Proposition \ref{KEY} relies on a careful bound for the operator $\T$, found in Lemma \ref{Tthm}.

\section{Main Results}\label{sec:main}

Let $\Gamma = \{\partial_t, \partial_x, L = t\partial_t+(x/a)\partial_x\}$ and let $Lu_0 = Lu (0,x)$. 
\begin{theorem}\label{growth}
Let $u(t,x) = e^{it|D|^a}u_0(x)$ with $L^iu_0\in \dot{H}^{\frac{1-a}{2}}$ for $i=0,1$. Then, for any time $t>0$
\begin{equation}
|u(t, yt^{1/a})|\leq C (1+|y|^{\frac{-a}{4(1-a)}})\mhalf{t}\left(\|u_0\|_{\dot{H}^{\frac{1-a}{2}}}+\|Lu_0\|_{\dot{H}^{\frac{1-a}{2}}}\right). \end{equation}
\end{theorem}
This theorem is a special case of the following proposition combined with Lemma \ref{Sob}:
\begin{ppn}\label{KEY}
Let $u(t,x) = e^{it|D|^a}u_0(x)$ with $L^iu_0\in \dot{H}^{\frac{1-a}{2}}$ for $i=0,1$. Then we have the following restricted $L^2$ bounds:
\begin{enumerate}
\item For $0< a$ and $a\neq 1$ with \[\sigma\in\left\{\begin{array}{ll}
                                                   \left(0,\frac{1-a}{2}\right]& \text{ for }0<a<1\\
						   \left[\frac{1-a}{2},0\right) & \text{ for } 1<a,
                                                      \end{array}\right.\] 
\begin{equation}\label{hom}
\int_T^{2T} |u(t,yt^{1/a})|^2dt \leq C|y|^{\frac{1-a-2\sigma}{a-1}}T^{\frac{a-1+2\sigma}{a}}\left(1 + y^{\frac{a}{2(a-1)}}\right) \|u_0\|^2_{\dot{H}^\sigma}. 
\end{equation}
\item For $0<a<1$ and $\frac{1-a}{2}\leq\sigma<1/2$,
 \begin{equation}\label{inhom}
\int_T^{2T} |u(t,yt^{1/a})|^2dt \leq C|y|^{\frac{1-a-2\sigma}{a-1}}T^{\frac{a-1+2\sigma}{a}}\left(1 + y^{\frac{a}{2(a-1)}}\right) \|u_0\|^2_{H^\sigma}.                                                      
                                                     \end{equation}

\end{enumerate}
\end{ppn} 

In the results above, there is a growth factor in $y$. In the following optimality result, we show it is not possible to remove the growth factor without slowing down the rate of decay.
\begin{theorem}\label{sharp}
Choose initial data $u_0$ such that $\widehat{u}_0(\xi) = |\xi|^{a-1}\widehat{g\varphi}(y\xi+|\xi|^a)$ with $\|g\|_{L^2}=1$ and $\varphi\in C_0^\infty(\R)$ such that $0\leq \varphi$ and $\varphi^2(t)\leq \chi_{[T,2T]}(t)$. Then, 
\begin{equation*}
\left(\int_T^{2T}|u(t,yt)|^2dt\right)^{1/2}\geq C'(a)|y|^{a/4(a-1)}C(g)\|u_0\|_{\dot{H}^{(1-a)/2}} 
\end{equation*} where the constant $C(g)$ is explicitly \[C(g)^2 =\begin{cases} 
\displaystyle{\int}^{(a-1)|y/a|^{a/(a-1)}}_0 |\zeta|^{-1/2}|\widehat{g\varphi}(\zeta-(a-1)|y/a|^{a/(a-1)})|^2d\zeta&1<a\\
\displaystyle{\int}_{(a-1)|y/a|^{a/(a-1)}}^0 |\zeta|^{-1/2}|\widehat{g\varphi}(\zeta+(1-a)|y/a|^{a/(a-1)})|^2d\zeta&0<a<1\end{cases}.\] \end{theorem}

We can remove the factor of $y$ in the Theorem \ref{growth} above but at the cost of a slower decay rate. One such result is the following:
\begin{ppn} \label{inftydecay}
\begin{enumerate}
 \item When $0<a<1$, $u(t,x) = e^{it|D|^a}u_0(x)$ with $L^iu_0\in H^{r}$ with $r=\max\{\frac{1-a}{2}, \frac{2-a}{4}\}$ for $i=0,1$ satisfies the following $L^\infty$ bound:
\[\sup_{y}|u(t, yt^{1/a})|\leq C(t^{-1/4}+t^{-1/2})\sum_{|k|\leq 1}\left(\|L^k u_0\|_{
H^{\frac{1-a}{2}}}+\|L^k u_0\|_{H^{\frac{2-a}{4}}}\right).\]
\item When $2\leq a$, $u(t,x) = e^{it|D|^a}u_0(x)$ with $L^iu_0\in \dot{H}^{r}$ with $r=\max\{\frac{1-a}{2}, \frac{2-a}{4}\}$ for $i=0,1$ satisfies the following $L^\infty$ bound:
\[\sup_y |u(t, yt^{1/a})|\leq C(t^{-1/4}+t^{-1/2})\sum_{|k|\leq 1}\left(\|L^k u_0\|_{\dot{H}^{\frac{1-a}{2}}}+\|L^k u_0\|_{\dot{H}^{\frac{2-a}{4}}}\right).\]
\end{enumerate}
\end{ppn}
This proposition follows from an invariant vector field Sobolev bound (Lemma \ref{Sob}) and Proposition \ref{KEY} for appropriate choices of $\sigma$.  Then, the proofs of Theorem \ref{growth} and Proposition \ref{inftydecay} are straightforward applications of Lemma \ref{Sob} and Proposition \ref{KEY}. 

\subsection{Proof details}\label{dtls}

Most of the detailed work is in the proof of Proposition \ref{KEY} which has several steps. 
 
First, we present the proofs of Theorem \ref{growth} and Proposition \ref{inftydecay} assuming Proposition \ref{KEY}, followed by the longer proof of Proposition \ref{KEY}. 

\subsubsection{Proof of Theorem \ref{growth} and Proposition \ref{inftydecay}}\label{pf1}

For these two short proofs, we will assume Proposition \ref{KEY}.
\begin{proof}[Proof of Theorem \ref{growth}]
Let $\mathcal{S}$ and $\T$ be the operators described above. Then, 
\begin{align}
\left(\int_T^{2T}\right.&\left.\left|u (t, y t^{1/a})\right|^2dt\right)^{\frac{1}{2}} \leq\half{\left(\int_T^{2T} \left|\mathcal{S}u_0(t)\varphi_T(t)\right|^2dt\right)}\nonumber\\
&=\sup_{\|g\|_{L^2}=1}\left|\langle \widehat{u}_0, \T\left(g\varphi_T\right)\rangle\right|\nonumber\\
&\leq\|u_0\|_{\dot{H}^{\frac{1-a}{2}}}\sup_{\|g\|_{L^2}=1}\half{\left(\int\left|\int e^{-iy\tau^{1/a} \xi-i\tau|\xi|^a} g(\tau)\varphi_T(t)d\tau\right|^2 |\xi|^{a-1}d\xi\right)}\label{ybe}.\end{align}
Then by Lemma \ref{Tthm}, 
\begin{align*} 
\sup_{\|g\|_{L^2}=1}\half{\left(\int\left|\int e^{-iy\tau^{1/a} \xi-i\tau|\xi|^a} g(\tau)\varphi_T(t)d\tau\right|^2 |\xi|^{a-1}d\xi\right)}&\leq\sup_{\|g\|_{L^2}=1} \half{\left(1+y^{\frac{a}{2(a-1)}}\right)}\|g\|_{L^2}\\
 &\leq C(a)  \half{\left(1+y^{\frac{a}{2(a-1)}} \right)}.
\end{align*}
Thus, 
\[\left(\int_T^{2T}\left|u (t, y t^{1/a})\right|^2dt\right)^{\frac{1}{2}} \leq C(a)  \half{\left(1+y^{\frac{a}{2(a-1)}} \right)} \|u_0\|_{\dot{H}^{\frac{1-a}{2}}}. \]
\end{proof}

In order to get the slower decay rate, we use a similar argument, but we use a linear combination of results from Proposition \ref{KEY} with $\sigma = \frac{1-a}{2}$, and $ \sigma= \frac{2-a}{4}$.   
\begin{proof}[Proof of Proposition \ref{inftydecay}]
Recall from Proposition \ref{KEY},  for $0<a<1$ and $\frac{1-a}{2}\leq\sigma<1/2$, \[\int_T^{2T} |u(t,yt^{1/a})|^2dt \leq C\left(|y|^{\frac{1-a-2\sigma}{a-1}}T^{\frac{a-1+2\sigma}{a}} + |y|^{\frac{2-a-4\sigma}{2(a-1)}}T^{\frac{a-1+2\sigma}{a}}\right) \|u_0\|^2_{H^\sigma}.\]
We want to choose discrete $\sigma$ so that we can control the right hand side independent of $y$. Notice that 
 \[|y|^{\frac{1}{a-1}} T^{-1/a}<1\Rightarrow |y|^{\frac{1}{a-1}}<T^{1/a}\Rightarrow |y|^{\frac{a}{2(a-1)}}<T^{1/2}.\] So when $|y|^{\frac{1}{a-1}} T^{-1/a}<1$, if we take $\sigma = \frac{1-a}{2}$, we have:
\[\int_T^{2T} |u(t,yt^{1/a})|^2dt \leq C\left(1 + T^{1/2}\right) \|u_0\|^2_{H^{\frac{1-a}{2}}}.\]
On the other hand, if 
\[|y|^{\frac{1}{a-1}} T^{-1/a}>1 \Rightarrow |y|^{-\frac{a}{2(a-1)}}T^{1/2}<1.\] If we take $\sigma = \frac{2-a}{4}$, then $|y|^{\frac{1-a-2\sigma}{a-1}}T^{\frac{a-1+2\sigma}{a}} = |y|^{-\frac{a}{2(a-1)}}T^{1/2} $ and we get
 \[\int_T^{2T} |u(t,yt^{1/a})|^2dt \leq C\left(1 + T^{1/2}\right) \|u_0\|^2_{H^{\frac{2-a}{4}}}.\] Therefore, we can conclude that 
\[\int_T^{2T} |u(t,yt^{1/a})|^2dt \leq C\left(1 + T^{1/2}\right)\left(\|u_0\|^2_{H^{\frac{1-a}{2}}}+ \|u_0\|^2_{H^{\frac{2-a}{4}}}\right).\] By combining this estimate and the Sobolev estimate from Lemma \ref{Sob}, we get the desired $L^\infty$ bound. 
\end{proof}

\subsubsection{Reduction to Lemma \ref{Tthm}}\label{pf2}

The proofs of Theorem \ref{growth} and Proposition \ref{inftydecay} rely on Proposition \ref{KEY}, which in turn follows from a proposition on the Fourier transform of the dual operator $\mathcal{S}^*$. Recall the operator $\mathcal{S}$:
\[\mathcal{S} v(t) = \int e^{i(yt^{1/a}\xi +t|\xi|^a)}\widehat{v}(\xi)d\xi.\]  
Then, 
\begin{align}\label{Tcalc}
\int_T^{2T} |u(t,yt^{1/a})|^2dt &\leq \sup_{g\in L^2, \|g\|=1}|\langle\mathcal{S} u_0(t)\varphi_T(t), g(t) \rangle|\nonumber\\
&= \sup_{g\in L^2, \|g\|=1}|\langle \widehat{u_0}(\xi), \T(\varphi_T g)(\xi) \rangle|\\
&= \sup_{g\in L^2, \|g\|=1}\half{\left(\int \frac{1}{\omega(\xi)}\left|\widehat{u_0}(\xi)\right|^2d\xi\right)} \half{\left(\int \omega(\xi)|\T(\varphi_T g)(\xi) |^2d\xi\right)}\nonumber
\end{align} where $\omega(\xi)$ is a weight function such as $|\xi|^{-2\sigma}$ or $(1+\xi^2)^{-\sigma}$.

Thus, we have reduced the proof of Proposition \ref{KEY} to showing the weighted estimate contained in the following lemma.
\begin{lem}\label{Tthm}
Let $g\in L^2(\R)$ and $\T$ and $\varphi_T$ be as above. 
\begin{enumerate}
\item For $0< a$ and $a\neq 1$ with \[\sigma\in\left\{\begin{array}{ll}
                                                   \left(0,\frac{1-a}{2}\right]& \text{ for }0<a<1\\
						   \left[\frac{1-a}{2},0\right) & \text{ for } 1<a,
                                                      \end{array}\right.\] 
\begin{equation}\label{homT}
\int |\xi|^{-2\sigma} |\T(g\varphi_T)(\xi)|^2d\xi \leq C |y|^{\frac{1-a-2\sigma}{a-1}}T^{\frac{a-1+2\sigma}{a}}\left(1 + y^{\frac{a}{2(a-1)}}\right)\|g\|_{L^2}^2. 
\end{equation}

\item For $0<a<1$ and $\frac{1-a}{2}\leq\sigma<1/2$, 
\begin{equation}\label{inhomT}
\int (1+\xi^2)^{-\sigma} |\T(g\varphi_T)(\xi)|^2d\xi \leq C|y|^{\frac{1-a-2\sigma}{a-1}}T^{\frac{a-1+2\sigma}{a}}\left(1 + y^{\frac{a}{2(a-1)}}\right)\|g\|_{L^2}^2. 
\end{equation}

\end{enumerate}
\end{lem}

The proof of Proposition \ref{KEY} follows by applying Lemma \ref{Tthm} to the last line of \eqref{Tcalc}. 

\begin{rem}
The lemma above is the crux of the argument in that it tracks the contributions of the exponent $a$ and the parameter $y$ on the bounds of the operators. In practice, these are oscillatory integral operator bounds, 
similar to those seen in Miyachi's work \cite{Miyachi} and developed in many different directions. We consider oscillatory integrals of the form 
\[\int e^{i\Psi(\xi)}w(\xi)d\xi\] where the weight function $w$ corresponds to the homogeneous or inhomogeneous Sobolev weight. The primary technical difference between the work here and in early work such as Miyachi 
is we do not assume compact support in the oscillatory integral. The balance between behavior at the origin, near the critical points, and at infinity gives the range of spaces listed in Lemma \ref{Tthm}. 
\end{rem}

\subsubsection{Proof of Lemma \ref{Tthm}} \label{mainT}

In the case of general $a$, the proof of Lemma \ref{Tthm} relies on careful analysis of the kernels of the various Fourier integral operators. Before we begin the proof, we collect a few useful results from harmonic analysis. First, we will use the Hardy-Littlewood-Sobolev inequality to control the kernel of the operator near critical points of the phase function.  
\begin{lem}[Hardy-Littlewood-Sobolev Lemma]\label{HLS}
For $n\geq 1$, $1<p<q<\infty$, $0<\beta<n$, and  \begin{equation}
I_\beta g(x)  = \int_{\R^n} \frac{g(z)}{|x-z|^{\beta}}dz,                                                                                         
                                                                                        \end{equation}
 \[\|I_\beta g\|_{L^q(\R^n)}\leq C(p,q) \|g\|_{L^p(\R^n)}\mbox{ when  }\,\frac{1}{q} = \frac{1}{p} -\frac{n-\beta}{n}.\] 
\end{lem}
For a proof of this lemma, see \cite{Stein1993}. 

The other major pieces of the kernels we study are easily dealt with using the following straightforward corollaries to the $T1$ Theorem of David and Journ\'e \cite{Journe1984}. These corollaries allow us to deal with a very specific example of the so-called standard kernels, either with bounds independent of an additional variable $\xi$ or with an $L^1$ factor in $\xi$.
\begin{ppn}\label{noxik}
Let $a(t,s;\xi)$ be smooth function in $t$ and $s$ supported in ball of radius $\rho$ in $\R_t\times\R_s$ such that $a(t,s;\xi) = a(s,t;\xi)$ and both $a(t,s;\xi)$ and $\partial_t a(t,s;\xi)$ are uniformly bounded in $t,s$ and $\xi$ by constants $C_1$ and $C_2$, respectively. Define the kernel $k(t,s;\xi) = a(t,s;\xi)(t-s)^{-1}$. Then $k(t,s;\xi)$ is a standard kernel uniformly bounded in $\xi$ and the operator $T$ associated to $k(t,s;\xi)$ is bounded from $L^2$ to itself independent of $\xi$ with norm $C_1 +\rho C_2$. 
\end{ppn}
\begin{ppn}\label{parxik}
Let $A(t,s;\xi)$ be smooth function in $t$ and $s$ supported in ball of radius $\rho$ in $\R_t\times\R_s$ and $k(\xi)\in L^1(\R)$ with $k>0$. Assume that $A(t,s;\xi) = A(s,t;\xi)$ and $|A(t,s;\xi)|\leq C_1 k(\xi)$ and $|\partial_t A(t,s;\xi)|\leq C_2 k(\xi)$. Define the kernel $K(t,s;\xi) = A(t,s;\xi)(t-s)^{-1}$. Then $K(t,s;\xi)$ is a standard kernel and the operator $T$ associated to $K(t,s;\xi)$ is bounded from $L^2$ to itself with $\|T\|_{2,2}\leq C k(\xi)$. 
\end{ppn}
\begin{rem}
 These propositions are easily shown using the $T1$ theorem, so we state them here without proof. 
\end{rem}

\begin{proof}[Proof of Lemma \ref{Tthm}]
 Assume without loss of generality that $y>0$. The two cases in $a$ are proved using the same techniques, but the details are subtly different. We will present the full details for the $0<a<1$ for \eqref{inhomT} and a sketch of the ideas for \eqref{homT}. The arguments are identical for the $1<a$ case. 

The main idea of this proof is rewriting the weighted $L^2$ norm as an operator and analyzing the kernel of this operator. We will show that the kernel is a linear combination of standard kernels and fractional integrals. Since the kernel is an oscillatory integral, we will use a careful decomposition combined with integration by parts and the method of stationary phase to control it. 

We begin by rewriting the square on the left hand side of the inequality as a product and changing the order of integration. In the proofs of \eqref{homT} and \eqref{inhomT}, the only differences lie in the kernel analysis. For now, let $\omega(\xi)$ denote a general weight function. Then, 
\[\int \omega(\xi)|\T(g\varphi_T)(\xi)|^2d\xi = \iint g\varphi_T(t) \overline{g\varphi_T}(s) \int e^{-i\Psi(\xi)}\omega(\xi)d\xi dtds,\] where $\Psi(\xi) = y(t^{1/a} - s^{1/a})\xi +(t-s)|\xi|^a$ and let 
$\xi_0 = -\left(\frac{yf(t,s)}{a}\right)^{\frac{1}{a-1}}$ denote the critical point of this phase function $\Psi(\xi)$ with $f(t,s) := (t^{1/a}-s^{1/a})/(t-s)$.  Our goal is to show kernel estimates on the $\xi$ 
integral so that we can apply H\"older and produce $L^2$ bounds. For clarity, let $K(t,s) = \int e^{-i\Psi(\xi)}\omega(\xi)d\xi$. To get reasonable bounds on $K(t,s)$, we need to consider the integral near the 
critical point and away from the critical point. With different weight functions, these estimates proceed somewhat differently. Consider $\omega(\xi) =|\xi|^{-2\sigma}$.
%
 Let $\e \leq |\xi_0|/2$; then:
\begin{align*}
 K(t,s)& = \int_{|\xi-\xi_0|<\e} e^{-i\Psi(\xi)}(1+\xi^2)^{-\sigma}d\xi + \int_{|\xi-\xi_0|>\e} e^{-i\Psi(\xi)}(1+\xi^2)^{-\sigma}d\xi\\
&= I+II.
\end{align*}
Now, $|I|\leq 2C\e (1+\xi_0^2)^{-\sigma}$, and we will use integration by parts to bound the second term: 
\begin{align*}
 II &= \int_{|\xi-\xi_0|>\e} \partial_\xi \left(e^{-i\Psi(\xi)}\right)\frac{1}{-i(1+\xi^2)^{\sigma}\Psi'(\xi)}d\xi\\
& = \left.\frac{e^{-i\Psi(\xi)}}{-i\Psi'(\xi)(1+\xi^2)^{\sigma}}\right|_{|\xi-\xi_0|>\e} - \int_{|\xi-\xi_0|>\e} e^{-i\Psi(\xi)}\partial_\xi \left(\frac{1}{-i(1+\xi^2)^{\sigma}\Psi'(\xi)}\right)d\xi 
\end{align*} Here the requirements for $\sigma$ come into play. In order to have these terms be finite at zero and decay at $\infty$, $0<\sigma$. We will show the bounds for the boundary term in detail, and the bounds on the remaining term are similar. When we evaluate the boundary term, the only contributions come from $\xi_0-\e$ and $\xi_0+\e$, so we have 
\[BT = \frac{e^{-i\Psi(\xi_0-\e)}}{-i\Psi'(\xi_0-\e)(1+(\xi_0-\e)^2)^{\sigma}} -  \frac{e^{-i\Psi(\xi_0+\e)}}{-i\Psi'(\xi_0+\e)(1+(\xi_0+\e)^2)^{\sigma}}.\] First, notice that $\Psi'(\xi_0 \pm \e) = \pm a(1-a)\e(t-s)|\xi_0|^{a-2}$. We can neglect the higher order terms in $\e$ and rewrite the boundary terms: 
\begin{align*}
|BT|&\leq \frac{|\xi_0|^{2-a}}{a(1-a)\e|t-s|}\left|\frac{e^{-i\Psi(\xi_0-\e)}}{(1+(\xi_0-\e)^2)^{\sigma}} +  \frac{e^{-i\Psi(\xi_0+\e)}}{(1+(\xi_0+\e)^2)^{\sigma}}\right|\\
&\leq \frac{C|\xi_0|^{2-a}}{a(1-a)\e|t-s|(1+\xi_0^2)^\sigma}. 
\end{align*}
Now, we optimize our choice of $\e$ by setting the two terms equal to each other:
\begin{align*}
 2C\e (1+\xi_0^2)^{-\sigma} = & \frac{C|\xi_0|^{2-a}}{a(1-a)\e|t-s|(1+\xi_0^2)^\sigma}\\
\e^2  = &\frac{C|\xi_0|^{2-a}}{a(1-a)|t-s|}\\
\e = &\frac{C'|\xi_0|^{1-a/2}}{|t-s|^{1/2}}
\end{align*}

However, this optimal $\e_1 = \frac{C'|\xi_0|^{1-a/2}}{|t-s|^{1/2}}$ is not always less than $|\xi_0|/2$. Let \[\e = \min\left\{\frac{C'|\xi_0|^{1-a/2}}{|t-s|^{1/2}}, |\xi_0|/2\right\}.\] In the homogeneous case, 
a similar argument produces the same choice of $\e$.

We will do the rest of the calculations with $\omega(\xi) = (1+\xi^2)^{-\sigma}$. The homogeneous case follows by similar and slightly easier arguments.

Observe that
\begin{equation}\label{eps} |\xi_0|/2<\frac{C'|\xi_0|^{1-a/2}}{|t-s|^{1/2}} \Leftrightarrow y^{\frac{a}{2(1-a)}}> \frac{C|t-s|^{1/2}}{f(t,s)^{\frac{a}{2(1-a)}}}.\end{equation} Assume $y^{\frac{a}{2(1-a)}}>\frac{C|t-s|^{1/2}}{f(t,s)^{\frac{a}{2(1-a)}}}$, and therefore $\e = |\xi_0|/2$. Let $\gamma(a) = a^{\frac{1}{a-1}}$. For $0<a<1$, $\gamma(a)>2$. More importantly, $\Psi(\gamma(a)\xi_0)=0$. We adjust the decomposition of $K(t,s)$ so that one endpoint lies on this very convenient number: 
\begin{align*}
K(t,s)&=\int^{\xi_0/2}_{\gamma(a)\xi_0} e^{-i\Psi(\xi)}(1+\xi^2)^{-\sigma}d\xi + \int_{(\gamma(a)\xi_0,\xi_0/2)^C} e^{-i\Psi(\xi)}(1+\xi^2)^{-\sigma}d\xi\\
K(t,s)&= K_1(t,s) + K_2(t,s).
\end{align*}
The easier term is $K_1$, so we will bound it first. Clearly, \[|K_1(t,s)|< C(\gamma(a) - 1/2)|\xi_0|(1+\xi_0^2)^{-\sigma}.\] By our assumption on $y$, $|\xi_0|<\frac{C'|\xi_0|^{1-a/2}}{|t-s|^{1/2}}$, so we can apply the Hardy-Littlewood-Sobolev lemma, Lemma \ref{HLS}, for fractional integration: 
\begin{align*}
\left|\int g\varphi_T(t) \int K_1(t,s) \overline{g\varphi_T}(s)ds\,dt \right|&\leq \int\hspace{-1pt}|g\varphi_T(t)|\int \frac{C'(\gamma(a) - 1/2)|\xi_0|^{1- \frac{a}{2}}|g\varphi_T(s)|}{(1+\xi_0^2)^\sigma |t-s|^{1/2}}ds\,dt\\
&\leq \frac{C'(\gamma(a) -1/2)a^{(2-a)/2(1-a)} y^{\frac{2- a}{2(a-1)}} T^{\frac{a-2}{2a}}}{(1+C a^{\frac{2}{1-a}}y^{\frac{2}{1-a}}T^{-2/a})^\sigma }\|g\varphi_T\|^2_{L^{4/3}}\\
&\leq  C''(\gamma(a) -1/2)a^{\frac{2-a-4\sigma}{2(1-a)}} y^{\frac{2-a-4\sigma}{2(a-1)}} T^{\frac{a-2+4\sigma}{2a}}T^{1/2}\|g\|_{L^2}^2\\
&\leq C_1(a) y^{\frac{2-a-4\sigma}{2(a-1)}}T^{\frac{a-1+2\sigma}{a}}\|g\|_{L^2}^2.                                                                                                                                                                                                                                                                                                                                                                                                 \end{align*}
Now we turn to $K_2(t,s)$. Recall
\[K_2(t,s) = \int_{-\infty}^{\gamma(a)\xi_0} e^{-i\Psi(\xi)}(1+\xi^2)^{-\sigma}d\xi+\int_{\xi_0/2}^\infty e^{-i\Psi(\xi)}(1+\xi^2)^{-\sigma}d\xi.\] Let's begin with the first term. By integration by parts, 
\begin{align*}
 \int_{-\infty}^{\gamma(a)\xi_0}\frac{ e^{-i\Psi(\xi)}}{(1+\xi^2)^{\sigma}}d\xi& =\left. \frac{e^{-i\Psi(\xi)}}{-i\Psi'(\xi)(1+\xi^2)^\sigma}\right|_{-\infty}^{\gamma(a) \xi_0} \\
&\hspace{1in}- \int_{-\infty}^{\gamma(a)\xi_0} e^{-i\Psi(\xi)}\partial_\xi\left(\frac{1}{-i\Psi'(\xi)(1+\xi^2)^{\sigma}}\right)d\xi\\
& = \frac{|\xi_0|^{1-a}}{-ia(1-a)(t-s)(1+a^{\frac{2}{a-1}}\xi_0^2)^\sigma}\\
&\hspace{1in}- \int_{-\infty}^{\gamma(a)\xi_0} e^{-i\Psi(\xi)}\partial_\xi\left(\frac{1}{-i\Psi'(\xi)(1+\xi^2)^{\sigma}}\right)d\xi\\
&= K'_2(t,s) +K''_2(t,s).
\end{align*}
The kernel $K'_2(t,s)$ satisfies all the conditions of Proposition \ref{noxik} with $C_1 = \frac{y^{-1}T^{1- 1/a}}{(1-a)(1+y^{2/(a-1)}T^{-2/a})^\sigma}$, $C_2 =\frac{2y^{-1}T^{-1/a}}{(1-a)(1+y^{2/(a-1)}T^{-2/a})^\sigma}$ and $\rho =T$. Therefore, 
\begin{align*}
| \int g\varphi_T(t) \int K'_2(t,s) \overline{g\varphi_T}(s)ds |&\leq \|g\|_{L^2} \frac{Cy^{-1}T^{1- 1/a}}{(1-a)(1+y^{2/(a-1)}T^{-2/a})^\sigma}\|g\|_{L^2} \\
 &\leq \frac{C}{1-a} y^{\frac{1-a-2\sigma}{a-1}}T^{\frac{2\sigma-1+a}{a}}\|g\|_{L^2}^2
\end{align*}

For the kernel $K_2 ''(t,s)$, we want to change the order of integration, so that we may integrate in $t$ and $s$ before integrating in $\xi$. Showing $K''_2$ is a nice kernel is complicated by the presence of the exponential; if we change the order of integration, the exponential splits into a function of norm 1 in $t$ and $s$, leaving only the derivative for the kernel. In fact, $\partial_\xi\left(\frac{1}{-i\Psi'(\xi)(1+\xi^2)^{\sigma}}\right)$ satisfies the kernel conditions of Proposition \ref{parxik}. It will be convenient for notation to let $\xi_2 =-(\frac{y}{a} (2T)^{(1-a)/a})^{\frac{1}{a-1}})$ and $\xi_1 = -(\frac{y}{a} T^{(1-a)/a})^{\frac{1}{a-1}})$ Since $\gamma(a)\xi_0 < \xi_2$, when we change the order of integration, we have 
\begin{align*}
\int&\int g\varphi_T (t)\overline{g\varphi_T}(s) \int_{-\infty}^{\gamma(a)\xi_0} e^{-i\Psi(\xi)}\partial_\xi\left(\frac{1}{-i\Psi'(\xi)(1+\xi^2)^{\sigma}}\right)d\xi\\
&\\
& = \int_{-\infty}^{\xi_2}\iint_{\Omega(t,s)(\xi)} g(t)e^{-i(yt^{1/a}\xi+t|\xi|^a)} \frac{\varphi_T(t) \overline{\varphi_T}(s)A(t,s;\xi)}{i(t-s)}\overline{g}(s)e^{i(ys^{1/a}\xi+s|\xi|^a)}ds\,dt\,d\xi \end{align*}
where \[A(t,s;\xi) =\varphi_T(t) \varphi_T(s) \partial_\xi \left(-i(yf(t,s) +a|\xi|^{a-1}\sgn\xi)(1+\xi^2)^{\sigma}\right)^{-1},\] and $\Omega(t,s)(\xi)$ is the region in $\R^2$ from Fubini theorem. When $\xi < \xi_1$, we can apply Proposition \ref{parxik} for $(t,s)\in [T,2T]^2$, so the exact description of $\Omega(t,s)$ is not important, since its intersection with the square is clearly contained in the square. So we can decompose again:
\begin{align*}
 \int_{-\infty}^{\xi_2} &\iint_{\Omega(t,s)(\xi)} g(t)e^{-i(yt^{1/a}\xi+t|\xi|^a)} \frac{\varphi_T(t) \overline{\varphi_T}(s)A(t,s;\xi)}{i(t-s)}\overline{g}(s)e^{i(ys^{1/a}\xi+s|\xi|^a)} ds dt d\xi\\& = \int_{-\infty}^{\xi_1} \iint_{\Omega(t,s)(\xi)} g(t)e^{-i(yt^{1/a}\xi+t|\xi|^a)} \frac{\varphi_T(t) \overline{\varphi_T}(s)A(t,s;\xi)}{i(t-s)}\overline{g}(s)e^{i(ys^{1/a}\xi+s|\xi|^a)} ds dt d\xi\\&+\int_{\xi_1}^{\xi_2} \iint_{\Omega(t,s)(\xi)} g(t)e^{-i(yt^{1/a}\xi+t|\xi|^a)} \frac{\varphi_T(t) \overline{\varphi_T}(s)A(t,s;\xi)}{i(t-s)}\overline{g}(s)e^{i(ys^{1/a}\xi+s|\xi|^a)} ds dt d\xi\\
&=i+ii
\end{align*}
We will use the following claim, combined with Proposition \ref{parxik} to bound term $i$: 
\begin{clm}
Let \[k(\xi) = \frac{1}{(1+\xi^2)^\sigma (yT^{1/a -1}/a - a|\xi|^{a-1})}.\] Then, for $A(t,s;\xi)$ defined above and $\xi\in(-\infty,\xi_1)$, we have
\[|A(t,s;\xi)|\leq \frac{2a}{1-a}\left(1+ \frac{a(2^{(1-a)/a}-1)}{1 -a }\right)k'(\xi) \] and \[|\partial_t A(t,s;\xi)| = |\partial_s A(t,s;\xi)| \leq \frac{4a}{T(1-a)^2}\left(1+ \frac{a(2^{(1-a)/a}-1)}{1 -a }\right)k'(\xi).\]
\end{clm}
The proof of this claim is straightforward, since by definition \newline $yf(t,s) -a|\xi|^{a-1}>((1-a)/a)yT^{\frac{1 -a}{a}}$. Therefore, 
\begin{align*}
 |i| \leq \int_{-\infty}^{\xi_1} &\left|\iint_{\Omega(t,s)(\xi)} g(t)e^{-i(yt^{1/a}\xi+t|\xi|^a)} \frac{\varphi_T(t) \overline{\varphi_T}(s)A(t,s;\xi)}{i(t-s)}\overline{g}(s)e^{i(ys^{1/a}\xi+s|\xi|^a)} ds dt\right| d\xi\\
&\leq \|g\|_{L^2}\int_{-\infty}^{\xi_1} C(a) k'(\xi)d\xi \|g\|_{L^2}\\
&\leq\frac{aC(a)}{\left(1 + y^{2/(a-1)}T^{-2/a}\right)^\sigma(1-a)yT^{(1-a)/a}}\|g\|_{L^2}^2\\
&\leq aC(a) y^{\frac{1-a-2\sigma}{a-1}}T^{(2\sigma +a-1)/a}\|g\|^2_{L^2}
\end{align*}
where $C(a) = \frac{2a}{1-a}\left(1+ \frac{a(2^{(1-a)/a}-1)}{1 -a }\right)\left(1+\frac{2}{1-a}\right)$.

Now we consider term $ii$. Since we are integrating in $\xi$ over a bounded interval, we do not need to use Proposition \ref{parxik}. It suffices to show that $A(t,s;\xi)$ is bounded. With $\xi\in(\xi_1,\xi_2)$, we cannot neglect the region $\Omega(t,s)$ in favor of $[T,2T]^2$. Observe that $\Omega(t,s) = \{|\xi|^{a-1}< yf(t,s) < \frac{y}{a}T^{(1-a)/a}\}$. Therefore, $yf(t,s) + a |\xi|^{a-1}\sgn\xi> (1-a)|\xi|^{a-1}$. Using this bound and the range of $\xi$, we have the following claim:

\begin{clm}
 Let $A(t,s;\xi)$ be as above with $t,s \in \Omega(t,s)$ and $\xi\in(\xi_1,\xi_2)$. Then, 
\[|A|\leq C (2\sigma+a(1-a))y^{(-a-2\sigma)/(a-1)}T^{(a+2\sigma)/a}\] and \[|\partial_t A| = |\partial_s A| \leq \frac{4C(2\sigma+2a)}{T(1-a)^2} y^{(-a-2\sigma)/(a-1)}T^{(a+2\sigma)/a}.\]
\end{clm}
Now, we use Proposition \ref{noxik}:
\begin{align*}
 |ii|&\leq \int_{\xi_1}^{\xi_2}\left| \iint_{\Omega(t,s)(\xi)} g(t)e^{-i(yt^{1/a}\xi+t|\xi|^a)} \frac{\varphi_T(t) \overline{\varphi_T}(s)A(t,s;\xi)}{i(t-s)}\overline{g}(s)e^{i(ys^{1/a}\xi+s|\xi|^a)} ds dt\right| d\xi\\
&\leq \|g\|_{L^2}\int_{\xi_1}^{\xi_2} \left(C (2\sigma+a(1-a)) +\frac{4C(2\sigma+2a)}{(1-a)^2}\right) y^{(-a-2\sigma)/(a-1)}T^{(a+2\sigma)/a}\|g\|_{L^2}d\xi\\
&\leq C'\left( (2\sigma+a(1-a)) +\frac{4(2\sigma+2a)}{(1-a)^2}\right)y^{\frac{1-a-2\sigma}{a-1}}T^{(2\sigma+a-1)/a}\|g\|_{L^2}^2
\end{align*}

Finally, we address the term on $(\xi_0/2,\infty)$. The integral from $\xi_0/2$ to 0 is somewhat simpler. Since $0<\sigma<1/2$, we can integrate $|\xi|^{-2\sigma}$ explicitly:
\[\left|\int_{\xi_0/2}^0 e^{-i\Psi(\xi)}(1+\xi^2)^{-\sigma}d\xi\right|\leq C|\xi_0|^{1-2\sigma} \leq \frac{C'|\xi_0|^{1-a/2-2\sigma}}{|t-s|^{1/2}}.\] Therefore, 
\[\left|\iint g\varphi_T(t)\overline{g\varphi_T}(s) \int_{\xi_0/2}^0 e^{-i\Psi(\xi)}(1+\xi^2)^{-\sigma}d\xi ds\,dt\right| \leq C'y^{\frac{2-a-4\sigma}{a-1}}T^{\frac{2\sigma+a-1}{a}}\|g\|^2_{L^2}. \]

When we integrate the kernel from $(0,\infty)$, we use integration by parts as before, and the boundary terms contribute nothing. For the derivative term, we need to be more precise as $\partial_\xi \left((1+\xi^2)^{-\sigma}(yf(t,s)+a|\xi|^{a-1})^{-1}\right)$ changes sign. We do an additional decomposition to preserve monotonicity. Let $\widetilde{\xi}$ be the critical point of $\partial_\xi \left((1+\xi^2)^{-\sigma}(yf(t,s)+a|\xi|^{a-1})^{-1}\right)$. Then, 
 \begin{align*} 
 \iint &\frac{g\varphi_T(t)\overline{g\varphi_T}(s)}{i(t-s)}\int_0^\infty e^{-i\Psi(\xi)} \partial_\xi\left((1+\xi^2)^{-\sigma}(yf(t,s)+a|\xi|^{a-1})^{-1}\right)ds\,dt\,d\xi\\
 &= \iint g\varphi_T(t)\overline{g\varphi_T}(s) K_3(t,s)dtds+\iint g\varphi_T(t)\overline{g\varphi_T}(s) K_4(t,s)dtds
 \end{align*} where 
 \begin{align*}
  K_3(t,s) &= \int_0^{\widetilde{\xi} }e^{-i\Psi(\xi)} \partial_\xi\left((1+\xi^2)^{-\sigma}(yf(t,s)+a|\xi|^{a-1})^{-1}\right)d\xi\\
  K_4(t,s)&=\int_{\widetilde{\xi}}^\infty e^{-i\Psi(\xi)} \partial_\xi\left((1+\xi^2)^{-\sigma}(yf(t,s)+a|\xi|^{a-1})^{-1}\right)d\xi
 \end{align*}

We will present the argument for $K_3$; the analysis for $K_4$ follows the same arguments but with slightly different constants (independent of $y$ and $T$). 
For $K_3$, we apply Proposition \ref{parxik} with $k'(\xi) = \left|\partial_\xi\left((1+\xi^2)^{-\sigma}(CyT^{(1-a)/a}+a|\xi|^{a-1})^{-1}\right)\right|$. Let $\xi_3$ denote the zero of $k'(\xi)$. Then, 
\begin{align*}
\left|\iint\right.&\left.\phantom{\int} g\varphi_T(t)\overline{g\varphi_T}(s) K_3(t,s)dtds\right| \\
& \leq \|g\|_{L^2}^2 \int_0^\infty \left|\partial_\xi\left(\frac{(1+\xi^2)^{-\sigma}}{CyT^{\frac{1-a}{a}}+a|\xi|^{a-1}}\right)\right|d\xi\\
&\leq  \frac{2\| g\|_{L^2}^2}{(1+\widetilde{\xi}^2)^{\sigma}(CyT^{\frac{1-a}{a}}+a|\xi_3|^{a-1})}
\end{align*}
In order for $\widetilde{\xi}$ to be a zero of $k'(\xi)$, it must satisfy \[a(1-a)(1+\xi_3^2) = 2\sigma|\xi_3|^{3-a}(yT^{(1-a)/a}+a|\xi_3|^{a-1}).\] Therefore,   
\[\left|\iint \frac{g\varphi_T(t)\overline{g\varphi_T}(s)}{i(t-s)}\int_0^{\widetilde{\xi}}e^{-i\Psi(\xi)} \partial_\xi\left(\frac{(1+\xi^2)^{-\sigma}}{yf(t,s)+a|\xi|^{a-1}}\right)d\xi ds\,dt \right|\leq C\|g\|_{L^2}^2 |\xi_3|^{-2\sigma - a+1}.\]
Observe that $\xi_3>C|yT^{(1-a)/a}|^{\frac{1}{a-1}}$ for a constant depending only on $a$ and $\sigma$. When $2\sigma\geq 1-a$, that means this $\widetilde{\xi}$ term is bounded by $y^{(-2\sigma-a+1)/(a-1)}T^{(2\sigma+a-1)/a}$ and can be combined with other terms. 
    
Now, we collect all the terms above:
\begin{lem}\label{ybig}
When $y^{\frac{a}{2(1-a)}}>\frac{C|t-s|^{1/2}}{f(t,s)^{\frac{a}{2(1-a)}}}$ and $1-a\leq 2\sigma<1$,   
\begin{equation}
\int (1+\xi^2)^{-\sigma} |\T(g\varphi_T)(\xi)|^2d\xi \leq \left(C_1(a) y^{\frac{a}{2(a-1)}}+C_2(a))\right)y^{\frac{1-a-2\sigma}{a-1}}T^{\frac{a-1+2\sigma}{a}}\|g\|_{L^2}^2.
\end{equation} where $C_2(a)=\left(C'\left( (2\sigma+a(1-a)) +\frac{4(2\sigma+2a)}{(1-a)^2}\right)+aC(a)\right)$ and $C_1(a) = C''(\gamma(a) -1/2)a^{\frac{2-a-4\sigma}{2(1-a)}}$
\end{lem}

It remains to show the bound for $y^{\frac{a}{2(1-a)}}<C\frac{C|t-s|^{1/2}}{f(t,s)^{\frac{a}{2(1-a)}}}$. Let $F(t,s) = \frac{C|t-s|^{1/2}}{f(t,s)^{\frac{a}{2(1-a)}}}$. In this case, we only need to take the optimal $\e$:
\begin{align*}
 \left|\iint_{y^{\frac{a}{2(1-a)}}< F(t,s)} \right.&\left.g\varphi_T(t) \overline{g\varphi_T}(s) \int e^{-i\Psi(\xi)}(1+\xi^2)^{-\sigma}d\xi dtds\right|\\
&\leq \iint_{y^{\frac{a}{2(1-a)}}< F(t,s)} |g\varphi_T(t)| |\overline{g\varphi_T}(s)| \frac{|\xi_0|^{1-a/2}}{(1+\xi_0^2)^\sigma |t-s|^{1/2}}ds dt\\
&\leq \frac{C(yT^{1/a -1})^{(2-a)/2(a-1)}}{(1+y^{2/a-1}T^{-2/a})^\sigma} \|g\varphi_T\|^2_{L^{4/3}}\\
&\leq \frac{Cy^{(2-a)/2(a-1)}T^{-(2-a)/2a}T^{1/2}}{(1+y^{2/a-1}T^{-2/a})^\sigma}\|g\varphi_T\|^2_{L^{2}}\\
&\leq C y^{\frac{2-a-4\sigma}{2(a-1)}}T^{\frac{a-1+2\sigma}{a}}\|g\|_{L^2}^2.
\end{align*}

We combine all of these terms to see that for $1-a\leq2\sigma<1$:
\begin{equation*}
\int (1+\xi^2)^{-\sigma} |\T(g\varphi_T)(\xi)|^2d\xi \leq Cy^{\frac{1-a-2\sigma}{a-1}}T^{\frac{a-1+2\sigma}{a}}\left(1 + y^{\frac{a}{2(a-1)}}\right) \|g\|_{L^2}^2. 
\end{equation*}

\end{proof}

\subsubsection{Remarks on the proofs of \S \ref{pf1} and \S \ref{pf2}}

The proof of Lemma \ref{Tthm} for the case $a>1$ differs in a couple key ways, but otherwise follows the same general argument. First of all, the weight in the kernel is replaced by $|\xi|^{2\sigma}$. We need the homogeneous weight here because $1/\Psi'(\xi)$ at $0$ is not equal to zero, and therefore we accumulate additional powers of $y$ and $T$ either by evaluating the derivative at $0$, or from the kernel with the exponential that we worked so hard to avoid in the case $0<a<1$. Unfortunately, the acceptable range of $\sigma$ means we cannot prove Proposition \ref{inftydecay} when $1<a<2$. 

That said, the proof above for Proposition \ref{inftydecay} is also correct when $2 \leq a$. The steps in proof of Proposition \ref{Tthm} are the same, except the various bounds on the kernels are subtly different.  

\subsection{Optimality and Counterexamples}\label{subsec:oandc}

In Proposition \ref{inftydecay}, the factor of $|y|$ acts as a barrier to our optimal time decay rate. In the following results, we explore the precise nature of this impediment. There are several different ways to consider the singularity that appears. Firstly, we will look along slightly different trajectories and find a lower bound (enforcing the optimality of our results), however this results imposes strong conditions on the initial data. 
 
\subsubsection{Lower bounds}
 Instead of considering $u(t,yt^{1/a})$, we will look at $u(t,yt)$ which will simplify our calculations considerably. Since \[\half{\left(\int_T^{2T}|u(t,zt)|^2dt\right)}\leq \sup_{[T,2T]}|u(t,zt)|\] without the need for our Sobolev lemma, this choice makes sense. Let \[S'f(t) = \int e^{i(yt\xi+t|\xi|^a)} \widehat{f}(\xi)d\xi\] and \[T'g(\xi) = \int e^{-i(yt\xi+t|\xi|^a)}g(t)\varphi(t)dt\] where $\varphi$ is a positive function with compact support such that $\varphi^2 \leq \chi_{[T,2T]}$.  We will show precisely the following 
\begin{thm}
Choose initial data $u_0$ such that $\widehat{u}_0(\xi) = |\xi|^{a-1}\widehat{g\varphi}(y\xi+|\xi|^a)$ with $\|g\|_{L^2}=1$ and $\varphi\in C_0^\infty(\R)$ such that $0\leq \varphi\leq \chi_{[T,2T]}(t)$. Then, 
\begin{equation*}
\left(\int_T^{2T}|u(t,yt)|^2dt\right)^{1/2}\geq C'(a)|y|^{a/4(a-1)}C(g)\|u_0\|_{\dot{H}^{(1-a)/2}} 
\end{equation*} where the constant $C(g)$ is explicitly \[C(g)^2 =\begin{cases} 
\displaystyle{\int}^{(a-1)|y/a|^{a/(a-1)}}_0 |\zeta|^{-1/2}|\widehat{g\varphi}(\zeta-(a-1)|y/a|^{a/(a-1)})|^2d\zeta&1<a\\
\displaystyle{\int}_{(a-1)|y/a|^{a/(a-1)}}^0 |\zeta|^{-1/2}|\widehat{g\varphi}(\zeta+(1-a)|y/a|^{a/(a-1)})|^2d\zeta&0<a<1\end{cases}\] \end{thm}

This proof relies heavily on the fact that $\T'$ above is the Fourier transform up to a constant. Since this operator has this nice property, we will show a lower bound for $\int |\xi|^{a-1}|T'g(\xi)|^2d\xi$ in the following proposition

\begin{lem}\label{keysharp}
Without loss of generality, assume that $y>0$. Let $\xi_1 = -\left(y/a\right)^{1/(a-1)}$ and $\Xi_1 = (a-1)|\xi_1|^a$. For $g\in L^2\bigcap H^{(a-1)/2}$ and $\varphi\in C_0^\infty$ with $0\leq\varphi\leq \chi_{[T,2T]}$, we have the following inequalities:
\begin{enumerate}
\item When $1<a$,
\begin{align}
\int |\xi|^{a-1}|\widehat{g\varphi}(y\xi+|\xi|^a)|^2d\xi\geq \frac{1}{a}&\int_0^\infty \Psi(\zeta)|\widehat{g}(\zeta)|^2d\zeta \\
&+ C(a)|y|^{\frac{a}{2(a-1)}}\int^{\Xi_1}_0 \mhalf{|\zeta|}|\widehat{g\varphi}(\zeta-\Xi_1)|^2d\zeta\nonumber 
\end{align}
 where $\Psi(\zeta)$ is bounded by $\frac{1}{a-1}$ at 0 and tends to $2/a$ as $\zeta$ goes to infinity.
\item When $0<a<1$,  
\begin{align}
\int |\xi|^{a-1}|\widehat{g\varphi}(y\xi+|\xi|^a)|^2d\xi\geq \int& \Phi(\zeta)|\widehat{g\varphi}(\zeta)|^2d\zeta\\
& + C(a)|y|^{\frac{a}{2(a-1)}}\int_{\Xi_1}^0 \mhalf{|\zeta|}|\widehat{g\varphi}(\zeta-\Xi_1)|^2d\zeta \nonumber
\end{align}
 where $\Phi(\zeta)$ is bounded at 0 by $max\{1/a, 1/(1-a)\}$ and decays like $y^{-a}|\zeta|^{a-1}$. 
\end{enumerate}
\end{lem}

\begin{proof}[Proof of Theorem \ref{sharp}] 
To complete the proof of Theorem \ref{sharp}, we combine the results of Lemma \ref{keysharp} with an inner product. Let $\varphi\in C_0^\infty(\R)$ such that $0\leq \varphi\leq \chi_{[T,2T]}(t)$. Therefore, 
\begin{align*}
\left(\int_T^{2T}\right. &\left. \vphantom{\int_T^{2T}}|u(t,yt)|^2dt\right)^{\frac{1}{2}} \geq\half{\left(\int|S'u_0 (t) \varphi(t)|^2dt\right)}\\
 = &\sup_{\|h\|_{L^2}=1}|\langle S'u_0(t) \varphi(t),h(t)\rangle|\\
=& \sup_{\|h\|_{L^2} =1}\left|\int \widehat{u_0}(\xi) \widehat{h\varphi}(y\xi+|\xi|^a)d\xi\right|\\
\geq& \int |\xi|^{a-1}|\widehat{g\varphi}(y\xi+|\xi|^a)|^2 d\xi\\
\geq& \half{\left(\int |\xi|^{1-a}|\widehat{u_0}(\xi)|^2 d\xi\right)} \half{\left(C(a)y^{\frac{a}{2(a-1)}}\int^{\Xi_1}_0 |\zeta|^{-1/2}|\widehat{g\varphi}(\zeta-\Xi_1)|^2d\zeta\right)}\\
 =& C'(a)y^{a/4(a-1)}\|u_0\|_{\dot{H}^{(1-a)/2}}\half{\left(\int^{\Xi_1}_0 |\zeta|^{-1/2}|\widehat{g\varphi}(\zeta-\Xi_1)|^2d\zeta\right)}.
\end{align*}
The constant in $g\varphi$ is controlled (loosely) by $\|g\|_{L^{4/3}}$. 
\end{proof}

\begin{proof}[Proof of Lemma \ref{keysharp}]
We will treat the cases $a>1$ and $0<a<1$ separately.

\begin{case}[ $a>1$]

Notice that $|\xi|^{a-1} = \dfrac{\sgn\xi}{a}\left(\dfrac{d}{d\xi} \left( y\xi +|\xi|^a\right) - y\right)$, so we can change variables:
\begin{align*}
\int |\xi|^{a-1}&|\widehat{g\varphi}(y\xi+|\xi|)^a|^2d\xi= \int\dfrac{\sgn\xi}{a}\left(\dfrac{d}{d\xi} \left( y\xi +|\xi|^a\right) - y\right)|\widehat{g\varphi}(y\xi+|\xi|)^a|^2d\xi\\
&=\int_0^\infty \frac{\psi'(\xi)}{a}|\widehat{g\varphi}(\psi(\xi))|^2d\xi - \int_{-\infty}^0 \frac{\psi'(\xi)}{a}|\widehat{g\varphi}(\psi(\xi))|^2d\xi\\
&\hspace{.25in} -\frac{y}{a}\int\sgn\xi|\widehat{g\varphi}(y\xi+|\xi|^a)|^2d\xi \\
&=\frac{2}{a}\int_0^\infty |\widehat{g\varphi}(\zeta)|^2d\zeta -\frac{y}{a}\int\sgn\xi|\widehat{g\varphi}(y\xi+|\xi|^a)|^2d\xi.
\end{align*}
On the intervals $(0,\infty)$ and $(-\infty, -y^{\frac{1}{a-1}})$, $y+a|\xi|^{a-1}\sgn\xi\neq 0$, so we can change variables again as long as we keep track of the derivative factor: 
\[-\frac{y}{a}\int_0^\infty |\widehat{g\varphi}(y\xi+|\xi|^a)|^2\frac{y+a|\xi|^{a-1}\sgn\xi}{y+a|\xi|^{a-1}\sgn\xi}d\xi = -\int_0^\infty \frac{y}{a}\Phi_+(\zeta)|\widehat{g\varphi}(\zeta)|^2d\zeta.\] Observe that $\frac{y}{a}\Phi_+(\zeta)$ is bounded at 0 by $1/a$ and decays like $|\zeta|^{-(a-1)/a}$ at $\infty$ since $\zeta \sim |\xi|^a$ for large $\xi$. Similarly, on $(-\infty, -y^{\frac{1}{a-1}})$, we have:
\[\frac{y}{a}\int_{-\infty}^{-y^{\frac{1}{a-1}}} |\widehat{g\varphi}(y\xi+|\xi|^a)|^2\frac{y+a|\xi|^{a-1}\sgn\xi}{y+a|\xi|^{a-1}\sgn\xi}d\xi = \int_0^\infty \frac{y}{a}\Phi_-(\zeta)|\widehat{g\varphi}(\zeta)|^2d\zeta\] where $\frac{y}{a}\Phi_-(\zeta)$ is bounded as $\zeta \rightarrow 0^+$ by $1/(a(a-1))$ and decays like  $|\zeta|^{-(a-1)/a}$ at positive infinity. Combining all the inequalities so far we find:
\begin{align}
\int |\xi|^{a-1}|\widehat{g\varphi}(y\xi+|\xi|)^a|^2d\xi = \int_0^\infty& \left(\frac{2}{a} -\frac{y}{a}\Phi_+(\zeta) +\frac{y}{a}\Phi_-(\zeta)\right)|\widehat{g\varphi}(\zeta)|^2d\zeta \\
&+\frac{y}{a}\int_{-y^{\frac{1}{a-1}}}^0 |\widehat{g\varphi}(y\xi+|\xi|^a)|^2d\xi. \nonumber
\end{align}
  The first two terms on the right hand side are precisely the lower bounds given in the Proposition where $\Psi = \frac{2}{a} -\frac{y}{a}\Phi_+(\zeta) +\frac{y}{a}\Phi_-(\zeta)$. The growth factor in $y$ arises from the remaining term. 

In the interval $(-y^{\frac{1}{a-1}},0)$, the function $y\xi+|\xi|^a$ is nearly parabolic, so the natural change of variables is $y\xi+|\xi|^a +(a-1)|\xi_1|^a = (\eta-\xi_1)^2$. In particular, we will take $\eta = \eta(\xi) = \xi_1 +\sgn(\xi-\xi_1) \sqrt{y\xi+|\xi|^a +\Xi_1}$ so that $\dfrac{y-a|\xi|^{a-1}}{2(\eta-\xi_1)}\geq 0$. Let $J(\xi) = \dfrac{2(\eta(\xi)-\xi_1)}{y-a|\xi|^{a-1}}$. Then if we let introduce $\eta(\xi)$, we have:
\[\frac{y}{a}\int_{-y^{\frac{1}{a-1}}}^0|\widehat{g\varphi}(y\xi+|\xi|^a)|^2d\xi=\frac{y}{a}\int_{-y^{\frac{1}{a-1}}}^0 J(\xi)|\widehat{g\varphi}((\eta(\xi)-\xi_1)^2-\Xi_1)|^2\dfrac{d\xi}{J(\xi)}. \]
We will show that $J(\xi)$ is bounded below precisely by $C(a)y^{-1} |\xi_1|^{a/2}$ in Proposition \ref{Jacob}. Now, we use this bound and then change variables:
\begin{align*}
\frac{y}{a}\int_{-y^{\frac{1}{a-1}}}^0|\widehat{g\varphi}(y\xi+|\xi|^a)|^2d\xi \geq & C(a)y^{\frac{a}{2(a-1)}}\int_{-y^{\frac{1}{a-1}}}^0|\widehat{g\varphi}((\eta(\xi)-\xi_1)^2-\Xi_1)|^2\dfrac{d\xi}{J(\xi)}\\
=& C(a)y^{\frac{a}{2(a-1)}}\int_{\xi_1 -\sqrt{\ldots}}^{\xi_1+\sqrt{\ldots}} |\widehat{g\varphi}((\eta-\xi_1)^2-\Xi_1)|^2d\eta\\ 
=&C(a)y^{\frac{a}{2(a-1)}}\int_{-\sqrt{\Xi_1}}^{+\sqrt{\Xi_1}} |\widehat{g\varphi}(\eta^2-\Xi_1)|^2d\eta\\
=&2 C(a)y^{\frac{a}{2(a-1)}}\int_{0}^{+\sqrt{\Xi_1}} \frac{1}{2\eta}|\widehat{g\varphi}(\eta^2+\Xi_1)|^2 2\eta d\eta\\
=&C(a)y^{\frac{a}{2(a-1)}}\int_0^{\Xi_1} \zeta^{-1/2}|\widehat{g\varphi}(\zeta+\Xi_1)|^2 d\zeta
\end{align*} where the final step is the substitution $\zeta = \eta^2$.
Therefore, we have our lower bound in the case $a>1$.
\end{case}
\begin{case}[$1>a>0$]

The argument in this case follows the previous case, but we omit the initial change of variable. \begin{align*}
\int |\xi|^{a-1}|\widehat{g\varphi}(y\xi+|\xi|^a)|^2d\xi =& \int_0^\infty \hspace{-.125in}|\xi|^{a-1}|\widehat{g\varphi}(y\xi+|\xi|^a)|^2d\xi \\
&+ \int_{-\infty}^{-y^{\frac{1}{a-1}}}\hspace{-.25in} |\xi|^{a-1}|\widehat{g\varphi}(y\xi+|\xi|^a)|^2d\xi\\&+ \int_{-y^{\frac{1}{a-1}}}^0 |\xi|^{a-1}|\widehat{g\varphi}(y\xi+|\xi|^a)|^2d\xi\\
=&\int_0^\infty \frac{(y+a|\xi|^{a-1})}{y|\xi|^{1-a}+a}|\widehat{g\varphi}(y\xi+|\xi|^a)|^2d\xi\\
 &+ \int_{-\infty}^{-y^{\frac{1}{a-1}}} \frac{(y-a|\xi|^{a-1})}{y|\xi|^{1-a}-a}|\widehat{g\varphi}(y\xi+|\xi|^a)|^2d\xi\\
&+ \int_{-y^{\frac{1}{a-1}}}^0 |\xi|^{a-1}|\widehat{g\varphi}(y\xi+|\xi|^a)|^2d\xi\\
=& \int\Phi(\zeta)|\widehat{g\varphi}(\zeta)|^2d\eta + \int_{-y^{\frac{1}{a-1}}}^0 |\xi|^{a-1}|\widehat{g\varphi}(y\xi+|\xi|^a)|^2d\xi
\end{align*} where $\Phi(\zeta)$ tends to $1/a$ as $\zeta\rightarrow 0^+$, $1/(1-a)$  for $\zeta\rightarrow 0^-$ and decays like $y^{-a}|\zeta|^{a-1}$ as $|\zeta|\rightarrow\infty$. In fact, the first term is bounded below by the $H^{(a-1)/2}$ norm (notice this is the inhomogeneous Sobolev space) and above by $\max\{1/a, 1/(1-a)\}$ times the the $L^2$ norm of $g$. 

In the interval $(-y^{\frac{1}{a-1}},0)$, the function $y\xi+|\xi|^a$ is nearly parabolic, so the natural change of variables is $y\xi+|\xi|^a +\Xi_1 = -(\eta-\xi_1)^2$. In particular, we will take $\eta = \eta(\xi) = \xi_1 -\sgn(\xi-\xi_1) \sqrt{|\Xi_1|-y\xi+|\xi|^a}$ so that $\dfrac{y-a|\xi|^{a-1}}{2(\eta-\xi_1)}\geq 0$. Let $\mathcal{J}(\xi) = \dfrac{2(\eta(\xi)-\xi_1)}{y|\xi|^{1-a}-a}$. Then we have:
\[\int_{-y^{\frac{1}{a-1}}}^0|\xi|^{a-1}|\widehat{g\varphi}(y\xi+|\xi|^a)|^2d\xi=\int_{-y^{\frac{1}{a-1}}}^0\mathcal{J}(\xi)|\widehat{g\varphi}(-(\eta(\xi)-\xi_1)^2+|\Xi_1|)|^2\dfrac{d\xi}{\mathcal{J}(\xi)}. \]
 We will show that $\mathcal{J}\geq C(a)y^{\frac{a}{2(a-1)}}$ in proposition \ref{Jacob}. If we change variables in the remaining term: 
\begin{align*}
\frac{y}{a}\int_{-y^{\frac{1}{a-1}}}^0&|\widehat{g\varphi}(y\xi+|\xi|^a)|^2d\xi \\
\geq& C(a)y^{\frac{a}{2(a-1)}}\int_{-y^{\frac{1}{a-1}}}^0|\widehat{g\varphi}(-(\eta(\xi)-\xi_1)^2+|\Xi_1|)|^2\dfrac{(y-a|\xi|^{a-1})d\xi}{2(\phi(\xi)-\xi_1)}\\
=&C(a)y^{\frac{a}{2(a-1)}}\int_{\xi_1 -\sqrt{\ldots}}^{\xi_1+\sqrt{\ldots}}|\widehat{g\varphi}(-(\eta-\xi_1)^2+|\Xi_1|)|^2d\eta\\ 
=&C(a)y^{\frac{a}{2(a-1)}}\int_{-\sqrt{|\Xi_1|}}^{+\sqrt{|\Xi_1|}} |\widehat{g\varphi}(-\eta^2+|\Xi_1|^a)|^2d\eta\\
=&2C(a)y^{\frac{a}{2(a-1)}}\int_0^{\sqrt{|\Xi_1|}} \frac{1}{-2\eta}|\widehat{g\varphi}(-\eta^2+|\Xi_1|)|^2 (-2\eta) d\eta\\
 =&2C(a)y^{\frac{a}{2(a-1)}}\int_{\Xi_1}^0 \frac{1}{2|\zeta|^{1/2}}|\widehat{g\varphi}(\zeta - \Xi_1)|^2 d\zeta
\end{align*} where the final step is the substiution $\zeta = -\eta^2$.
Therefore, for the case $0<a<1$, we have proved our lower bounds. 
\end{case}
\end{proof}
\subsubsection{Final remarks}
It is worth noting that in the case $a=2$ and $a=1/2$, we can get precisely equality, rather than inequality. The Jacobian bounds that we proved in the previous section are not necessary in these two cases, as $J(\xi)$ and $\mathcal{J}(\xi)$ are constant in the case $a=2$ an $a=1/2$, respectively. In the case $a=2$ and $a=1/2$, respectively, we have:
\[\int |\xi| |\widehat{g\varphi}(y\xi+\xi^2)|^2 d\xi = \int_0^\infty|\widehat{g\varphi}(\zeta)|^2d\zeta +\frac{y}{2}\int_0^{y^2/4}|\zeta|^{-1/2}|\widehat{g\varphi}(\zeta - y^2/4)|^2d\zeta \]
and 
\[\int |\xi|^{-1/2} |\widehat{g\varphi}(y\xi+|\xi|^{1/2})|^2 d\xi = \int\frac{|\widehat{g\varphi}(\zeta)|^2d\zeta}{\half{(1/4 +y|\zeta|)}} +y^{-1/2}\int_{-(4y)^{-1}}^0|\widehat{g\varphi}(\zeta +(4y)^{-1})|^2\frac{d\zeta}{\half{|\zeta|}}.\]

In order to understand where precisely this undesirable decay in Theorem \ref{growth} is coming from in the case of the linearized water wave, we bound $u(t,zt)$ below by 
\[\frac{1}{\half{T}}\half{\left(\int_T^{2T}|u(t,zt)|^2dt\right)}\leq \sup_{[T,2T]}|u(t,zt)|.\] By understanding the left hand side of this inequality, we can see the precise nature of the growth factor in Proposition \ref{KEY}. From the explicit bound on the analogue to operator $\T$, we get
\begin{equation}
 \int_{-\infty}^\infty \mhalf{|\xi|}|\widehat{g\vphi}(z\xi+\half{|\xi|})|^2d\xi = \int_{-\infty}^\infty \frac{|\widehat{g\vphi}(\zeta)|^2}{\half{\left(\frac{1}{4}+|z|\zeta\right)}}d\zeta +2\int_0^{\frac{1}{4|z|}} \frac{|\widehat{g\vphi}(\zeta)|^2}{\half{\left|\frac{1}{4}-|z|\zeta \right|}}d\zeta
\end{equation}
 The singularity appears only in the second term at $1/4|z|$, suggesting that as $|z|$ (or $|y|$) gets small, the problem with the decay exists only at high frequencies. Since existing results for the water wave problem 

\section{The Linearized Water Wave Problem}\label{sec:lin}

The results from subsection \ref{subsec:oandc} suggest that the problematic regions for Theorem \ref{growth} are for initial data away from the origin in frequency. On the other hand, previously mentioned existing results such as \cite[Proposition 3.1]{Wu2009}) are, in some sense, only problematic for initial data with a contribution from low frequencies. We combine these results and show an improved decay rate for solutions of the linearized water wave problem by imposing further bounds on the initial data. 

Recall that we have reduced the problem to the interface, and the linearized form of the initital value problems is 
\begin{equation}\label{linWW}
 \left\{\begin{array}{l}
\partial_t^2 u +|D|u =0\\
u(0,x) = u_0(x)\\
u_t(0,x) = u_1(x).
        \end{array}
\right.
\end{equation}
To contextualize these results, we will draw precise comparisons with the existing work in \cite{Wu2009}.

\subsection{Comparison to existing results}
Where previous work was concerned with the removal of the troublesome quantity $\Omega_0$, it is advantageous to reintroduce it here. Let $\Gamma = \{\partial_t, \partial_x, L = \frac{t}{2}\partial_t+x\partial_x\}$ and $\Omega = x\partial_t +\frac{t}{2}\partial_x|D|^{-1}$. The following proposition is analogous to \cite[Proposition 3.1]{Wu2009} and is proved using techniques similar to those used by Klainerman for the wave equation. While $\Gamma$ and $\Omega$ here are specific to the case of the linearized water wave problem, 
a similar collection of vector fields exists for \eqref{linWW} and the proposition can be generalized to this class of equations. We focus on the linearized water wave problem for the time being.

\begin{lem}\label{Omega}
Let $u(t,x)$ be any real-valued function which decays at infinity. Then, for a multiindex $k = \{k_1, k_2, k_3\}$ and $\Gamma^k =\partial_t^{k_1} \partial_x^{k_2}L^{k_3} $, we have:
\begin{equation}
 |u(t,x)|\leq \frac{C}{\half{t}}\left(\sum_{1\leq |k|\leq 2}\|\Gamma^k u(t)\|_{L^2(\R_x)} +\sum_{|k|\leq 1} \|\Gamma^k \Omega u(t)\|_{L^2(\R_x)}\right). 
\end{equation}
\end{lem}
\begin{rem}
 The details of this proof are due to unpublished work of Sijue Wu. The argument is similar to that of Proposition 3.1 of \cite{Wu2009}.
\end{rem}

\subsubsection{Energy bounds}\label{antider}
In order to turn the Klainerman-type estimates into $L^\infty$ bounds on the solution in terms of the initial data, we use the energy estimates for \eqref{linWW}. 
\begin{lem}\label{energy}
Let $u(t,x)$ be a solution of \eqref{linWW} with $(u(0,x), u_t(0,x)) = (u_0(x), u_1(x))$ and $u_i\in \mathscr{S}(\R)$ for $i=0,1$. In addition, let $\Gamma = \{\partial_t, \partial_x, L = \frac{t}{2}\partial_t+x\partial_x\}$ and $\Gamma_1= \Gamma \cup\{\Omega = x\partial_t +\frac{t}{2}\partial_x|D|^{-1}\}$ and let $k$ be a multiindex. Define the energy functional as \[E[v](t) = \int |\partial_tv(t,x)|^2+|\half{|D|}v(t,x)|^2dx.\] Then, $E[u](t) = E[u](0)$ and 
\[\|\Gamma_1^k u(t)\|_{L^2}\leq \|\partial_t\mhalf{|D|}\Gamma_1^k u(0)\|_{L^2}+\|\Gamma_1^k u(0)\|_{L^2}.\]
\end{lem}
\begin{rem}
This equality holds for a variety of classes of initial data. However, considering the data in Schwartz class allows us to use density arguments when the natural space for the data appears in our analysis.  
\end{rem}

\begin{proof}
Since $\partial_t, \partial_x$, and $\Omega$ are invariant under the operator $\partial_t^2+|D|$ and $[\partial_t^2+|D|, L]=\partial_t^2+|D|$, we can bound $\|\Gamma^\alpha u\|_{L^2}$ using the energy:
\begin{align*}
 \|\Gamma^\alpha u(t)\|_{L^2}&\leq \half{E[\mhalf{|D|}\Gamma^\alpha u](t)}\\
&=\half{E[\mhalf{|D|}\Gamma^ku](0)} \leq \half{\left(\int |\partial_t\mhalf{|D|}\Gamma^\alpha u(0,x)|^2+|\Gamma^\alpha u(0,x)|^2dx\right)}.
\end{align*}
A similar calculation holds for $\Omega\Gamma^\alpha$. 
\end{proof} 

\begin{rem}\label{remOm}
It is worth noting that the bound on $\Omega u(t,x)$ is not ideal:
\begin{align*}
\|\Omega u(t)\|_{L^2} &\leq \|\partial_t|D|^{-1/2}\Omega u (0)\|_{L^2}+\|\Omega u (0)\|_{L^2}\\
 &\leq \| \mhalf{|D|} x|D|u_0\|_{L^2}+\|\partial_x|D|^{-1}u_0\|_{L^2}+\|xu_1\|_{L^2}.
\end{align*}
We expect that $|D|^{-1/2}u_1$ has roughly the same regularity as $u_0$.  However, the term involving $\mhalf{|D|}u_0$ requires regularity on the antiderivative of $u_0$. This issue is precisely what caused the dependence of the data in \cite{Wu2009} on initial height and energy as well as initial slope. However, if $\widehat{u_0}$ was supported outside a ball centered at zero, we could control the bad term by the $L^2$ norm of the data. 
\end{rem}

\subsubsection{$L^\infty$ decay for the Linearized Water Wave problem}
The combination of Lemma \ref{energy} and Lemma \ref{Omega} yields the following $L^\infty$ bound. 
\begin{ppn}\label{KlainT}
 Let $u(t,x)$ be a solution of \eqref{linWW} with \[(u(0,x), u_t(0,x)) = (u_0(x), u_1(x))\in\mathscr{S}(\R)\times\mathscr{S}(\R). \]
Then, 
\begin{align}\label{KlainE}
 |u(t,x)|\leq& \,\, \frac{C}{\half{t}}\left(\sum_{1\leq|k|\leq 2}(\|\partial_t\mhalf{|D|}\Gamma^ku(0)\|_{L^2}+\|\Gamma^k u(0)\|_{L^2})\right)\\
&+\frac{1}{\half{t}}\left(\sum_{|k|\leq 1}(\|\partial_t\mhalf{|D|}\Gamma^k\Omega u(0)\|_{L^{2}}+ \|\Gamma^k\Omega u(0)\|_{L^{2}})\right).\nonumber\end{align}
\end{ppn}

\begin{rem}
The inequality \eqref{KlainE} has concise notation but it obscures the precise bounds on the right hand side. Using commutators, we can write each of these sums explicitly. The first two terms contain $L^2$ bounds of derivatives up to first order and homogeneous operators (such as $x\partial_x$) of the initial data. More interesting are the bounds on the second two terms:
\begin{align}
&\label{Omd}\sum\|\Gamma^k\Omega u(0)\|_{L^{2}}\lesssim \|xu_1\|_{L^2}+\|(x\partial_x)(xu_1)\|_{L^2}+\|\partial_x|D|^{-1}u_0\|_{L^2}\\
&\hspace{1in}+\|x|D|u_0\|_{L^2}+\|x\partial_xu_1\|_{L^2}+\|u_1\|_{L^2}\nonumber\\
&\label{Omtd}\sum\|\partial_t\mhalf{|D|}\Gamma^k\Omega u(0)\|_{L^{2}}\lesssim \|x|D|\mhalf{|D|}u_0\|_{L^2}+\|\partial_x|D|^{-1}\mhalf{|D|}u_0\|_{L^2}\\\vspace{-5pc}
&\hspace{1.35in}+\|(x\partial_x)(\partial_x|D|^{-1})\mhalf{|D|}u_0\|_{L^2}+\|(x\partial_x)(x|D| )\mhalf{|D|}u_0\|_{L^2}
\nonumber\\
&\hspace{1.35in}+\|\partial_x|D|^{-1}u_1\|_{L^2}+\|\half{|D|}u_0\|_{L^2}+\||D|u_0\|_{L^2}\nonumber\\
&\hspace{1.35in}+\|\partial_x|D|^{-1}u_0\|_{L^2}+\|x\partial_x\partial_x|D|^{-1} u_0\|_{L^2}\nonumber\\
&\hspace{1.35in}+\|x|D|\mhalf{|D|}u_1\|_{L^2}+\|\partial_x|D|^{-1}\mhalf{|D|}u_1\|_{L^2}\nonumber.
\end{align}
The first four terms contain the troublesome $\mhalf{|D|}u_0$, as mentioned in Remark \ref{remOm}. 
\end{rem}

These results give a decay of $\mhalf{t}$ for certain classes of data. The inequality in Proposition \ref{KlainT} along with the observation about $\Omega u(t)$ in Remark \ref{remOm} suggests that for data bounded away from the origin in frequency, Proposition \ref{KlainT} gives the desired $\mhalf{t}$ decay. On the other hand, Theorem \ref{growth} for $a = \frac{1}{2}$ gives the desired decay in the low frequency regime. Theorem \ref{growth} implies $\mhalf{t}$ decay whenever $|y|\geq1$ or other constant, and thus only when $|y|<1$ do we have an undesirable decay rate. Previous sharpness results for that theorem also suggest that singularity comes from a singularity in norm around $1/y$ in frequency.  Combining these two observations suggests that Theorem \ref{growth} is the right choice for initial data concentrated in low frequency. 

\subsection{Analysis of Initial Data}

The argument above suggests that we should examine data supported away from the origin in frequency and data supported near the origin in frequency independently. We begin with the first of these cases. 

\subsubsection{Large frequency data}

If instead we consider data supported in $|\xi|>R$, we can conclude the following corollary to Proposition \ref{KlainT} :
\begin{corollary}
 For $w(t,x)$ a solution of \eqref{linWW} with $(\widehat{w_0}(\xi), \widehat{w_1}(\xi))$ each supported in $|\xi|\geq R$, we have
\begin{equation}\label{wbd}
\sup_{y} |w(t,yt^2)|\leq \frac{1+R^{-1/2}}{\half{t}} \left(\sum_{|k|\leq 2}(\|\mhalf{|D|}\Gamma^kw_1\|_{L^2}+\|\Gamma^k w_0\|_{L^2})\right).
\end{equation}
\end{corollary}
\begin{proof}
 The inequality above follows directly from the pointwise Klainerman bound, Proposition \ref{KlainT}:
\begin{align*}
 \sup_{}|w(t,yt^2)| \leq& \frac{1}{\half{t}}\left(\sum_{1\leq|k|\leq 2}(\|\mhalf{|D|}\Gamma^kw_1\|_{L^2}+\|\Gamma^k w_0\|_{L^2})\right.\\
&\left.+\sum_{|k|\leq 1}(\|\mhalf{|D|}\Gamma^k\Omega w_1\|_{L^{2}}+ \|\Gamma^k\Omega w_0\|_{L^{2}})\right).
\end{align*} Since $\widehat{w_0}$ is supported in $|\xi|>R$, we can bound $\|\Omega w_0\|$ by
\[\|\Omega w_0\|_{L^2} \leq \sum_{|k|= 1}\|\Gamma^k w_0\|_{L^2}+R^{-1/2}\|w_0\|_{L^2}. \] Then, the full bound on the first term is the equation in the statement of the proposition. 
\end{proof}

\subsubsection{Low frequency data}
\begin{ppn}
 Let $u(t,x)$ be a solution to the differential equation (\ref{linWW}) with initial data $u_i(x)\in\mathscr{S}(\R)$ such that $\supp \widehat{u_i}(\xi)\subset B_R(0)$. Then,when $|y|\leq (8R^{1/2}T)^{-1} $  \[\sup_{|y|\leq (8R^{1/2}T)^{-1}} |u(t,yt^2)|\leq \frac{1}{T^{\frac{1}{2}}}\sum_{|k|\leq 1}\left(\|L^k u_0\|_{\dot H^{1/4}}+\|L^k u_1\|_{\dot H^{-1/4}}\right).\]
\end{ppn}

This proposition follows from Lemma \ref{Sob} and this proposition on the $L^2$ norm:

\begin{ppn}\label{lowfreqG}
 Let $\widehat{v}\in C_0^\infty(\R)$ with $\supp \widehat{v} \subseteq B_R(0)$. Let $1/4 \leq \sigma\leq 1/2$. Then, for $y<Y = (8R^{1/2}T)^{-1}$, 
\[\|\mathcal{S}v\|_{L^2(T,2T)}\leq C \|v\|_{\dot{H}^{\sigma}}.\]
\end{ppn}
\begin{proof}
Let $\varphi(t)\in C^\infty_0(\R)$ with $\varphi =1$ for $t\in(1,2)$ and $\varphi = 0$ for $t\in (1/2, 5/2)^C$ and $\varphi_T(t) = \varphi(t/T)$, and let $\chi(\xi)$ be identically $1$  on $B_R(0)$ and $0$ on $B_{2R}(0)^C$. Now, 
\begin{align*}
 \|\mathcal{S}v\|_{L^2(T,2T)} &= \sup_{g\in L^2, \|g\|=1} \left|\langle \mathcal{S}v(t)\varphi_T(t), g(t)\rangle\right|\\
&= \sup_{g\in L^2, \|g\|=1} \left|\langle \widehat{v}(\xi), \T(g\varphi_T)(\xi)\rangle\right|\\
&\leq \|v\|_{\dot{H}^{\sigma}}\left(\int (1+|\xi|^2)^{-\sigma} |\chi(\xi)\T(g\varphi_T)(\xi)|^2d\xi\right)^{\frac{1}{2}}.
\end{align*} As before, we rewrite the operator squared as a product and reorder the integral:
\[\int |\chi(\xi)\T(g\varphi_T)(\xi)|^2d\xi=\int g\varphi_T(t)\overline{g\varphi_T}(s) \int e^{-i(y(t^2-s^2)\xi+(t-s)|\xi|^{1/2})}\frac{|\chi(\xi)|^2}{(1+|\xi|^2)^{\sigma}}d\xi\]
The stationary point of the oscillatory integral is at $-1/(4y^2(t+s)^2)$, so if $2R <1/(4y^2(t+s)^2)$, we can simply integrate by parts. Then, $y<Y$ implies that $\xi_0\in B_{2R}(0)^C$, and we can control the integral in $\xi$ by integration by parts:
\begin{align*}
\int &e^{-i(y(t^2-s^2)\xi+(t-s)|\xi|^{1/2})}|\chi(\xi)|^2\mhalf{|\xi|}d\xi = \int_{-2R}^{2R} \frac{-\mhalf{|\xi|}\partial_\xi\left(e^{-i(y(t^2-s^2)\xi+(t-s)\half{|\xi|})}\right)d\xi}{i(y(t^2-s^2)+(1/2)(t-s)\mhalf{|\xi|}\sgn\xi)}\\
&=\left.\frac{e^{-i(y(t^2-s^2)\xi+(t-s)\half{|\xi|})}|\chi(\xi)|^2}{-i(y(t^2-s^2)|\xi|^{1/2}+(1/2)(t-s)\sgn\xi)}\right|_{-\infty}^{\infty}\\
&-\int e^{-i(y(t^2-s^2)\xi+(t-s)\half{|\xi|})} \partial_\xi\left(\frac{|\chi(\xi)|^2}{-i(y(t^2-s^2)\half{|\xi|}+(1/2)(t-s)\sgn\xi)}\right)d\xi\\
& = \frac{4i}{t-s} -\int e^{-i(y(t^2-s^2)\xi+(t-s)\half{|\xi|})} \partial_\xi\left(\frac{\chi(|\xi)|^2}{-i(y(t^2-s^2)\half{|\xi|}+(1/2)(t-s)\sgn\xi)}\right)d\xi
\end{align*}
The remaining term is also bounded by $|t-s|^{-1}$. In fact, following the same arguments as an earlier proof we can show the second term is a standard kernel with constant independent of $y$ and we can use the $T1$ theorem to show it is the kernel of a bounded operator. In fact, all the necessary bounds on the kernels are independent of the size of the support of $\widehat{v}$. Therefore, we have: 
\[\left(\int |\chi(\xi)\T(g\varphi_T)(\xi)|^2d\xi\right)^{1/2} \leq C\|g\varphi_T\|_{L^2} \] with $C$ independent of $y$, $T$, and $R$.
\end{proof}

\subsection{An improved $L^\infty$ decay result}

Since the results from the invariant vector field-type bounds and those from the Klainerman type bounds have difficulty controlling the solution in different areas, we combine the two results in order to improve the decay rate. 
We combine Theorem \ref{growth} in the case $a=1/2$ with Proposition \ref{KlainT} to get the following theorem:
\begin{theorem}\label{514}
 Let $u(t,x)$ be a solution of $\partial_t^2 u +|D|u =0$ with $(u(0,x),u_t(0,x))=(u_0,u_1)\in \mathscr{S}(\R)\times\mathscr{S}(\R)$. Then, 
\begin{equation}
\sup_{y\in\R}|u(t,yt^2)|\leq \frac{1}{t^{5/14}}\sum_{|k|\leq 2}\left(\|\Gamma^k u_0\|_{L^2}+\|\Gamma^k\mhalf{|D|} u_1\|_{L^2}\right)
\end{equation}
\end{theorem}
\begin{proof}
Fix $t>1$. Let $\chi(\xi)$ denote the indicator function for the ball of radius $t^p$ centered at $0$. Then, let $v(t,x)$ be the solution to \[(\partial_t^2+|D|)v =0 \text{ with }(\widehat{v}(0,\xi), \partial_t\widehat{v}(0,\xi)) = (\widehat{u_0}\chi, \widehat{u_1}\chi)\] and $w(t,x)$ be the solution to \[(\partial_t^2+|D|)w =0 \text{ with }(\widehat{w}(0,\xi), \partial_t\widehat{w}(0,\xi)) = (\widehat{u_0}(1-\chi), \widehat{u_1}(1-\chi)).\] Notice that since all of these differential equations are linear, $u(t,x) = w(t,x) +v(t,x)$. Therefore, 
\begin{equation}
|u(t,yt^2)| \leq |w(t,yt^2)|+ |v(t,yt^2)|
\end{equation}
Since the initial data for the first term is bounded away from $0$ in frequency, we will use the pointwise bound:
\begin{align*}
|w(t,yt^2)| \leq& \frac{1}{\half{t}}\left(\sum_{1\leq|k|\leq 2}(\|\mhalf{|D|}\Gamma^kw_1\|_{L^2}+\|\Gamma^k w_0\|_{L^2})\right)\\
&+\frac{1}{\half{t}}\left(\sum_{|k|\leq 1}(\|\mhalf{|D|}\Gamma^k\Omega w_1\|_{L^{2}}+ \|\Gamma^k\Omega w_0\|_{L^{2}})\right)
\end{align*} Since $\widehat{w_0}$ is supported in $|\xi|>t^p$, we can bound $\|\Omega w_0\|$ by
\[\|\Omega w_0\|_{L^2} \leq \sum_{|k|= 1}\|\Gamma^k w_0\|_{L^2}+t^{-p/2}\|w_0\|_{L^2}. \] Then, the full bound on the first term is  
\begin{equation}\label{wbd2}
|w(t,yt^2)|\leq \frac{1+t^{-p/2}}{\half{t}} \left(\sum_{|k|\leq 2}(\|\mhalf{|D|}\Gamma^kw_1\|_{L^2}+\|\Gamma^k w_0\|_{L^2})\right).
\end{equation}

For the function $v(t,x)$, first observe that if $|y|>1$, the decay is $\mhalf{t}$. The choice of $1$ here is slightly arbitary; what will matter more is a lower bound on $|y|$ from the analysis of the critical point. Since $v(t,x)$ has initial data compactly supported on the Fourier transform side, for sufficiently small values of $|y|$, we also have $\mhalf{t}$ decay. Precisely, if $|y|\leq (8T^{p/2+1})^{-1} $ 
\[|v(t,yt^2)|\leq \frac{1}{\half{t}}\sum_{|k|\leq 1}\left(\|L^k v_0\|_{H^{\frac{1}{4}}}+\|L^k v_1\|_{H^{\frac{-1}{4}}}\right).\]
Notice that because we have compact support in the Fourier transform, we can rewrite the right hand side here as 
\[|v(t,yt^2)|\leq \frac{t^{p/4}}{\half{t}}\sum_{|k|\leq 1}\left(\|L^k v_0\|_{L^2}+\|\mhalf{|D|}L^k v_1\|_{L^2}\right).\]
If $(8t^{p/2+1})^{-1}<|y|<1$, we use Theorem \ref{growth} and have
\begin{equation}\label{vbd}
 |v(t,yt^2)|\leq \frac{(1+t^{\frac{p+2}{8}})t^{p/4}}{\half{t}}\sum_{|k|\leq 1}\left(\|L^k v_0\|_{L^2}+\|\mhalf{|D|}L^k v_1\|_{L^2}\right). 
\end{equation}
 Different values of $p$ will cause different terms to dominate. When $p>0$, the contribution from \eqref{wbd2} will be $\mhalf{t}$, but \eqref{vbd} decays like $t^{\frac{3p}{8}}t^{\frac{1}{4}}\mhalf{t}$. These cannot be equal for any positive value of $p$. On the other hand, if $-2<p<0$, we have $t^{-\frac{p+1}{2}}$ from \eqref{wbd2} and $t^{\frac{3p}{8}}t^{\frac{1}{4}}\mhalf{t}$ from \eqref{vbd}, which are equal for $p = -2/7$. 

Notice that choosing $p=-2/7$ improves the decay in the case $|y|>1$ to $T^{-4/7}$ times $L^2$ norms. By taking $p=-2/7$, we can conclude that \[\sup_{y\in\R}|u(t,yt^2)|\leq \frac{1}{t^{5/14}}\sum_{|k|\leq 2}\left(\|\Gamma^k u_0\|_{L^2}+\|\Gamma^k\mhalf{|D|} u_1\|_{L^2}\right).\]
\end{proof}

\begin{rem}
Observe that in almost every term we can get the desired decay. In the case of $|\xi|>T^p$, whenever $p>0$, we get better than $\mhalf{T}$ decay, but at the cost of severely worse decay in the $|\xi|<T^p$ part. We might as well decompose around $|\xi| \sim1$, which gives the desired decay from the Klainerman type bounds with the smallest penalty on the remainder. In that remainder, only certain values of $|y|$ contribute to the growth, namely $(8T)^{-1}<|y|<1$. It is worth noting that this range is barely larger than the region described by the optimal choice $p = -2/7$, where we have $(8T)^{-6/7}<|y|<1$, but the miniscule reduction in the range of $y$ introduces $T^{1/7}$ of growth on the Klainerman term. Clearly there is more to understand with data compactly supported in frequency. 
\end{rem}

\subsection{Low frequency data: a second attempt}
What truly matters in this regime is whether or not the initial data has a singularity at the origin and how rapidly that singularity grows as the frequency approaches $0$. 

\begin{theorem}\label{odddecay}
 Let $u(t,x)$ be a solution of
\begin{equation}\label{halfWW}
\left\{\begin{array}{rl}
\partial_t u -i\half{|D|}u&=0\\
u(0,x) &= u_0(x)
 \end{array}
\right. 
\end{equation}
with $u_0(x)\in \dot{H}^{\frac{1}{4}}$ and $\supp \widehat{u_0}(\xi)\subseteq (-1,1)$ and 
$\widehat{u_0}= \mhalf{|\xi|}\sgn\xi |\xi|^\gamma \widehat{\psi}(\xi)$ for some compactly supported $\psi$.  In addition, let $\frac{C}{t}\leq |y| \leq 1$. Then 
\begin{equation}
 u(t,x) \in L^q(\R_x) \mbox{ where }\left\{\begin{array}{cc}
                    q\in \left(\frac{2}{1-2|\gamma|},\infty\right) & \mbox{ when } \gamma <0\\
		    q\in (2,\infty) &\mbox{ when } \gamma >0.
                    \end{array}.\right. 
\end{equation} 
\end{theorem}
\begin{rem}
 Heuristically, we expect $u(t,x)\in L^q(\R)$ to decay like $|x|^{-\frac{1}{q}}$. In the case of $-1/4<\gamma<0$, the reduced range of $q$ gives $|u(t,x)|$ would decay no faster than $|x|^{-\frac{1}{2}+|\gamma|}$, which would prevent $u(t,x)$ from being in $L^2$. In fact, if we consider $|u(t,yt^2)|$ and $\gamma = -1/4$, we would get that $|u(t,yt^2)|\leq y^{-\frac{1}{4}}\half{t}$, precisely the growth factor from our previous results. These heuristics suggest that the size of the singularity at the origin in frequency is what generates the troublesome growth factor. On the other hand, it appears that the solution to this issue is to consider data in $L^2$, which is an improvement over previous results. 
\end{rem}

The $L^q(\R)$ bounds on $u(t,x)$ are uniform in compact sets of $t$. We relate the function $u(t,x)$ to $u(t,yt^2)$ to take advantage of scaling in $t$ but return to $u(t,x)$ after manipulations have recast the problem in a nicer form. 

Theorem \ref{odddecay} follows from this proposition relating $u(t,yt^2)$ to a singular integral of $\psi$. 
\begin{ppn}
 Let $u$ be a solution to \eqref{halfde} with $u_0\in\dot{H}^{\frac{1}{4}}(\R)$ and $\supp\widehat{u_0}(\xi)\subseteq(-1, 1)$. Then, 
\begin{equation}
 |u(t,yt^2)| \leq C\int\frac{1}{\half{|y-z|}}\left|t\half{|D|}\mathcal{H}u_0(t^2z)\right|dz.
\end{equation} 
\end{ppn}

\begin{proof}
Since $u_0\in \dot{H}^{\frac{1}{4}}(\R)$ and is compactly supported in frequency, $\widehat{u_0}$ must take the following form for $-1/4<\gamma$ and $\psi\in C_0^\infty(\R)$ with $\supp\psi\subseteq (-1,1)$:
\begin{equation}\label{u0def}
\widehat{u_0}= \mhalf{|\xi|}\sgn\xi |\xi|^\gamma \widehat{\psi}(\xi).
\end{equation} The range of $\gamma$ is easily deduced by considering where the $\dot{H}^{\frac{1}{4}}$ norm of $u_0$ is finite. We make no additional assumptions on $\psi$. 

To prove the theorem, we begin by rewriting the solution in a different form. Observe that 
\begin{align*}
 u(t,yt^2) &= \int e^{i(yt^2\xi +t\half{|\xi|}) }\widehat{u_0}(\xi)d\xi\\ 
&= \int e^{i(y\xi +\half{|\xi|}) }\mhalf{|\xi|}\sgn\xi \half{|\xi|}\sgn\xi\frac{1}{t^2} \widehat{u_0}\left(\frac{\xi}{t^2}\right)d\xi\\ 
 &=\int e^{i(y\xi +\half{|\xi|}) }\mhalf{|\xi|}\sgn\xi \int \delta(\xi-\eta)\half{|\eta|}\sgn\eta\frac{1}{t^2} \widehat{u_0}\left(\frac{\eta}{t^2}\right)d\eta d\xi\\
&=\int e^{i(y\xi +\half{|\xi|})} \mhalf{|\xi|}\sgn\xi \iint Ce^{-iz(\xi-\eta)}dz\half{|\eta|}\sgn\eta\frac{1}{t^2} \widehat{u_0}\left(\frac{\eta}{t^2}\right)d\eta d\xi\\
&=\iint e^{i((y-z)\xi +\half{|\xi|})} \mhalf{|\xi|}\sgn\xi d\xi \int  Ce^{iz\eta}\half{|\eta|}\sgn\eta\frac{1}{t^2} \widehat{u_0}\left(\frac{\eta}{t^2}\right)d\eta dz\\
&= \int k(y-z) t\half{|D|}\mathcal{H} u_0(t^2z)dz.\end{align*} The second line comes from rescaling in $t$ and
 $k(y-z) = \int e^{i((y-z)\xi +\half{|\xi|})} \mhalf{|\xi|}\sgn\xi d\xi$. 

 To complete the proof of this theorem, we simply need to show that $|k(y-z)|\leq C\mhalf{|y-z|}$. We will conduct our analysis on $k(x)$ for simplicity in notation. Now, if we change variables $\zeta = \half{|\xi|}$,  
\begin{align}
k(x) &= \int e^{i(x\xi +\half{|\xi|})} \mhalf{|\xi|}\sgn\xi d\xi\nonumber\\
&= 2\int_0^\infty e^{i(x\zeta^2 +\zeta)}d\zeta - 2\int_0^\infty e^{-i(x\zeta^2 -\zeta)}d\zeta\nonumber\\
&=2\int_0^\infty e^{i\zeta}\left(e^{ix\zeta^2} - e^{-ix\zeta^2}\right)d\zeta.\label{oddk}
\end{align}
From \eqref{oddk}, it is clear that $k(-x)=-k(x)$. Since the phase functions in the calculation above are quadratic, it is possible to solve exactly for pieces of the kernel. Now,  
\begin{align*}
\sgn x k(x)&= 2e^{\frac{-i}{4|x|}}\int_0^\infty e^{i|x|\left(\zeta +\frac{1}{2|x|}\right)^2}d\zeta - 2e^{\frac{i}{4|x|}}\int_0^\infty e^{-i|x|\left(\zeta -\frac{1}{2|x|}\right)^2}d\zeta\\
&=2e^{\frac{-i}{4|x|}}\int_{\frac{1}{2|x|}}^\infty e^{i|x|\zeta^2}d\zeta - 2e^{\frac{i}{4|x|}}\int_{-\frac{1}{2|x|}}^\infty e^{-i|x|\zeta ^2}d\zeta\\
&= - 2e^{\frac{i}{4|x|}}\int_{-\infty}^\infty e^{-i|x|\zeta ^2}d\zeta +2e^{\frac{i}{4|x|}}\int^{-\frac{1}{2|x|}}_{-\infty} e^{-i|x|\zeta ^2}d\zeta+2e^{\frac{-i}{4|x|}}\int_{\frac{1}{2|x|}}^\infty e^{i|x|\zeta^2}d\zeta \\
&=-k_1(|x|) +k_2(|x|).
\end{align*}
The term $k_1(|x|)$ is straightforward to solve exactly using contours:
\begin{align*}
-2e^{\frac{i}{4|x|}}\int_{-\infty}^\infty e^{-i|x|\zeta ^2}d\zeta &= 2e^{\frac{i}{4|x|}}\int_{e^{\frac{-i\pi}{4}}\R} e^{-i|x|\zeta ^2}d\zeta\\
&= 2e^{\frac{i}{4|x|}}e^{-\frac{i\pi}{4}}\int_{-\infty}^{\infty} e^{-|x|\zeta ^2}d\zeta\\
&=\frac{ 2\sqrt\pi e^{\frac{i}{4|x|}}e^{-\frac{i\pi}{4}}}{\half{|x|}} . 
\end{align*}
Therefore, $|k_1(|x|)|\leq 2\sqrt\pi \mhalf{|x|}$. It only remains to control the second term, $k_2(|x|)$. Observe that $k_2(|x|)$ is the sum of integral and its complex conjugate, so it is sufficient to consider just one of the terms and show it is bounded by $C\mhalf{|x|}$. By construction, we can use integration by parts on the terms of $k_2$ as they avoid the critical point of the phase function. Thus,
\begin{align*}
\left| 2e^{\frac{-i}{4|x|}}\int_{\frac{1}{2|x|}}^\infty e^{i|x|\zeta^2}d\zeta \right| & = 2\left|\int_{\frac{1}{2|x|}}^\infty e^{i|x|\zeta^2}d\zeta\right|\\
&= 2\left|\int_{\frac{1}{2|x|}}^\infty \frac{\partial_\zeta\left(e^{i|x|\zeta^2}\right)}{2i|x|\zeta}d\zeta\right|\\
&=2\left|\left.\frac{e^{i|x|\zeta^2}}{2i|x|\zeta}\right|_{\frac{1}{2|x|}}^\infty + \int_{\frac{1}{2|x|}}^\infty \frac{e^{i|x|\zeta^2}}{2i|x|\zeta^2}d\zeta\right| \leq 4
\end{align*}
If $|x|\leq \pi/16$, this calculation implies that $k_2(|x|)$ is bounded by $C\mhalf{|x|}$. In order to see the behavior of this term for large $x$, we will calculate it using contours as we did for $k_1$. Then, 
\begin{equation}
 2e^{\frac{-i}{4|x|}}\int_{\frac{1}{2|x|}}^\infty e^{i|x|\zeta^2}d\zeta   = 2e^{\frac{-i}{4|x|}}\left(e^{\frac{i\pi}{4}}\int_{\frac{1}{2|x|}}^\infty e^{-|x|\zeta^2}d\zeta +\int_{\Gamma} e^{i|x|\zeta^2}d\zeta \right)
\end{equation}
where $\Gamma =\{|\zeta| = \frac{1}{2|x|}, \theta\in(0,\pi/4)\}$. By a similar argument similar to the one used for $k_1$, the first term is bounded by $e^{\frac{-1}{8|x|}}\mhalf{|x|}$. The second term can be rewritten as an integral in $\theta$ and bounded like so:
\begin{align*}
\left| \int_{\Gamma} e^{i|x|\zeta^2}d\zeta \right| &=\left|\int_{0}^{\frac{\pi}{4}} e^{\frac{ie^{2i\theta}}{4|x|}}\frac{ie^{i\theta}}{2|x|}d\theta \right|\\
&\leq \frac{\pi}{8|x|}
\end{align*}
 If $|x|>\frac{\pi}{16}$, we get precisely that the sum of these two terms is less than $2\sqrt\pi \mhalf{|x|}$. Therefore, 
\begin{equation}
 |k(x)|\leq C\mhalf{|x|}.
\end{equation} and we can conclude that 
\begin{equation}
 |u(t,yt^2)|\leq  C\int\frac{t|\half{|D|}\mathcal{H}u_0(t^2z)|}{\half{|y-z|}}dz .
\end{equation}
\end{proof}
Given the relationship between $u(t,yt^2)$ and the fractional integral of $\half{|D|}\mathcal{H}u_0$, the proof of Theorem \ref{odddecay} reduces to careful application of the Hardy-Littlewood-Sobolev lemma. 
\begin{proof}[Proof of Theorem \ref{odddecay}]
By a straightforward change of variables,
\[C\int\frac{t|\half{|D|}\mathcal{H}u_0(t^2z)|}{\half{|y-z|}}dz = C\int\frac{|\half{|D|}\mathcal{H}u_0(z)|}{\half{|yt^2-z|}}dz.\] Let 
$F(x) = \int\frac{|\half{|D|}\mathcal{H}u_0(z)|}{\half{|x-z|}}dz$. Observe that $|u(t,x)|\leq |F(x)|$, independent of $t$. 

Now we apply the Hardy-Littlewood-Sobolev lemma (Proposition \ref{HLS}) and get $\|F\|_{L^q}\leq \|\half{|D|}\mathcal{H}u_0\|_{L^p}$ for $\frac{1}{q} = \frac{1}{p} -\frac{1}{2}$. To complete the proof, it suffices to identify to which $L^p$ spaces $\half{|D|}\mathcal{H}u_0$ belongs. 

Recall $\widehat{u_0}$ is of the form \eqref{u0def}. When $\gamma>0$, $\half{|D|}\mathcal{H}u_0$ is some order derivative of $\psi$. Since $\widehat{\psi}\in C_0^\infty$, we know that $\psi$ is in Schwartz class, and thus $\half{|D|}\mathcal{H}u_0$ is in $L^p$ for all $p$. Then, we can conclude that $F\in L^q$ for all $q\in(2,\infty)$. 

When $-\frac{1}{4}<\gamma<0$, we will need to apply the Hardy-Littlewood-Sobolev lemma for a second time. By definition, $\half{|D|}\mathcal{H}u_0(z) = I_{1-|\gamma|}\psi (z)$, so we know that  for $1<r<p<\infty$
\[\|\half{|D|}\mathcal{H}u_0\|_{L^p}\leq \|\psi\|_{L^r}\mbox{ when  }\frac{1}{p} = \frac{1}{r} -\frac{|\gamma|}{1}.\]
By combining the bound on $F$ with the bound on $\half{|D|}\mathcal{H}u_0$, we have 
\[\|F\|_{L^q}\leq \|\half{|D|}\mathcal{H}u_0\|_{L^r}\mbox{ for }\frac{1}{q} = \frac{1}{r}-|\gamma| -\frac{1}{2}.\] 
Since $\psi$ is Schwartz, we know that $\psi\in L^r$ for $1\leq r\leq \infty$. In order to satisfy the lemma, we limit $r$ to the range $(1, \frac{1}{\frac{1}{2}+|\gamma|})$, which implies that $q\in (\frac{2}{1- 2|\gamma|},\infty)$
\end{proof}

In addition to the upper bounds found above, we also have the following lower bounds for more specialized data. 
 
\subsection{Sharp lower bounds for low frequency data}
\begin{theorem}
Let $u$ be a solution to the initial value problem
\[\left\{\begin{array}{rl}
\partial_t u -i\half{|D|}u&=0\\
u(0,x) &= u_0(x)
 \end{array}
\right.\] where $\widehat{u_0}(\xi) = \mhalf{|\xi|}\sgn\xi \widehat{\varphi}(\xi)$ and $\varphi(x)$  compactly supported in the interval $(-\frac{1}{M},\frac{1}{M})$, $M\in\mathbb{N}$ and $\varphi(x)$ does not change sign. Let $\frac{C}{t} \leq |y|\leq \delta$ where $\delta>0$ and  independent of $t$. In addition, assume that $\delta +\frac{1}{M}\leq \frac{\pi}{64}$. Then, 
\begin{equation}
 |u(t,yt^2)|\geq \frac{\sqrt\pi}{2}\int \frac{t|\varphi(t^2 z)|}{\half{|y-z|}}dz.
\end{equation}
\end{theorem}

\begin{rem}
 The sign assumption on $\varphi$ is a technical condition which allows us to move absolute values inside the integral without changing the value. It may be possible to avoid this condition through other techniques. 
\end{rem}
\begin{proof}
We will show that given sufficient large $M$, this part of the kernel convolved with the scaled initial data bounds the solution below. To begin, we verify that $\frac{1}{2}|k_1|>|k_2|$.  Observe that $k_2(|x|)$ is the sum of integral and its complex conjugate, so it is sufficient to consider just one of the terms and show it is bounded by $\frac{1}{4}|k_1|$. Then,
\begin{align*}
\left| 2e^{\frac{-i}{4|x|}}\int_{\frac{1}{2|x|}}^\infty e^{i|x|\zeta^2}d\zeta \right| & = 2\left|\int_{\frac{1}{2|x|}}^\infty e^{i|x|\zeta^2}d\zeta\right|\\
&= 2\left|\int_{\frac{1}{2|x|}}^\infty \frac{\partial_\zeta\left(e^{i|x|\zeta^2}\right)}{2i|x|\zeta}d\zeta\right|\\
&=2\left|\left.\frac{e^{i|x|\zeta^2}}{2i|x|\zeta}\right|_{\frac{1}{2|x|}}^\infty + \int_{\frac{1}{2|x|}}^\infty \frac{e^{i|x|\zeta^2}}{2i|x|\zeta^2}d\zeta\right|\leq 4.
\end{align*} As long as $|x|\leq \pi/4$, $k_1$ is the dominant part of the kernel. In our convolution operator, $x = y-z$ with $|z|\leq \frac{1}{Mt^2}$, so $|y-z|\leq |y|+|z| \leq \delta +\frac{1}{Mt^2}$. By our assumption on $M$ and $\delta$ , $k_1$ is the dominant part of the kernel. 

Now, since \begin{align*}
|u(t,yt^2)| &= \left|\int k(y-x) t\varphi(t^2z)dz\right| \\
&\geq \left|\int k_1(y-x) t\varphi(t^2z)dz\right| - \left|\int k_2(y-x) t\varphi(t^2z)dz\right|,            
           \end{align*}
in order to show we have a lower bound we also need that
\[\left|\int k_2(y-x) t\varphi(t^2z)dz\right|\leq \frac{1}{2}\left|\int k_1(y-x) t\varphi(t^2z)dz\right|.\]  In fact we will show first that $\left|\int k_1(y-x) t\varphi(t^2z)dz\right|$ is bounded below and then show that $\left|\int k_2(y-x) t\varphi(t^2z)dz\right|$ is bounded by half of this lower bound for the $k_1$ term. 

\begin{clm}
Assume that $\varphi$ does not change sign (that is, it is strictly positive or negative inside its support). For $k_1(x)$ defined as above, we have 
\begin{equation}\label{k1bd}
\left|\int k_1(y-z) t\varphi(t^2z)dz\right|\geq \sqrt\pi\int \frac{t|\varphi(t^2z)|}{\half{|y-z|}}dz.
\end{equation}
\end{clm}

\begin{proof}
Recall that $k_1(y-z) = \frac{ 2\sqrt\pi e^{\frac{i}{4|y-z|}}e^{-\frac{i\pi}{4}}\sgn(y-z)}{\half{|y-z|}}.$ Then, 
\begin{align*}
 \int k_1(y-z) t\varphi(t^2z)dz =&\int \frac{ 2\sqrt\pi e^{\frac{i}{4|y-z|}}e^{-\frac{i\pi}{4}}\sgn(y-z)}{\half{|y-z|}} t\varphi(t^2z)dz \\
=&\int \frac{ 2\sqrt\pi e^{\frac{i}{4|y|}}e^{-\frac{i\pi}{4}}\sgn y}{\half{|y-z|}} t\varphi(t^2z)dz\\&+\int \frac{ 2\sqrt\pi \left(e^{\frac{i}{4|y-z|}}\sgn(y-z)- e^{\frac{i}{4|y|}}\sgn y\right)e^{-\frac{i\pi}{4}}}{\half{|y-z|}} t\varphi(t^2z)dz\\
 \left|\int k_1(y-z) t\varphi(t^2z)dz \right|\geq&\left|\int \frac{ 2\sqrt\pi e^{\frac{i}{4|y|}}e^{-\frac{i\pi}{4}}\sgn y}{\half{|y-z|}} t\varphi(t^2z)dz\right|\\
&-\left|\int \frac{ 2\sqrt\pi \left(e^{\frac{i}{4|y-z|}}\sgn(y-z)- e^{\frac{i}{4|y|}}\sgn y\right)e^{-\frac{i\pi}{4}}}{\half{|y-z|}} t\varphi(t^2z)dz\right|\\
\geq&\left|\int \frac{ 2\sqrt\pi t\varphi(t^2z)}{\half{|y-z|}} dz\right| \\
&-\left|\int \frac{ 2\sqrt\pi \left(e^{\frac{i}{4|y-z|}}\sgn(y-z)- e^{\frac{i}{4|y|}}\sgn y\right)e^{-\frac{i\pi}{4}}}{\half{|y-z|}} t\varphi(t^2z)dz\right|
\end{align*}
Since $\varphi$ doesn't change sign, we can move the absolute value inside the first term, and it is sufficient to show 
\begin{equation}\label{k1sts}
\left|\int \frac{ 2\sqrt\pi \left(e^{\frac{i}{4|y-z|}}\sgn(y-z)- e^{\frac{i}{4|y|}}\sgn y\right)e^{-\frac{i\pi}{4}}}{\half{|y-z|}} t\varphi(t^2z)dz\right| \leq \int \frac{ \sqrt\pi t|\varphi(t^2z)|}{\half{|y-z|}} dz.
\end{equation}
To show \eqref{k1sts}, we will use the assumption of compact support for $\varphi$. Assume that $2\delta > \frac{1}{M}$. Then, $\sgn(y-z) = \sgn y$. By the mean value theorem, we have
\[e^{\frac{i}{4|y-z|}}\sgn(y-z)- e^{\frac{i}{4|y|}}\sgn y = \frac{-ize^{\frac{i}{4p(z)}}}{p(z)^2}\]
for $p(z)\in (y-z, y+z)$ gives the correct point in this interval for the Mean Value Theorem. Then, 
 \begin{align*}
  \left|\displaystyle{\int}\right.&\left. \frac{ 2\sqrt\pi \left(e^{\frac{i}{4|y-z|}}\sgn(y-z)- e^{\frac{i}{4|y|}}\sgn y\right)e^{-\frac{i\pi}{4}}t\varphi(t^2z)dz}{\half{|y-z|}} \right| \\
&\leq\left|\int \frac{ -2i\sqrt\pi ze^{\frac{i}{4p(z)}} e^{-\frac{i\pi}{4}}}{p(z)^2\half{|y-z|}} t\varphi(t^2z)dz\right|\\
&\leq\int \frac{ 2\sqrt\pi |z|}{|p(z)|^2\half{|y-z|}} t|\varphi(t^2z)|dz
 \end{align*}
Since $|z| \leq \frac{1}{Mt^2}$, $p(z) \in (y-z, y+z)$, and $y>\frac{1}{t}$, we have \[\frac{|z|}{p(z)^2}\leq \frac{\frac{1}{Mt^2}}{|y-\frac{1}{Mt^2}|^2}\leq \frac{Mt^2}{|Mt^2y -1|^2}\leq \frac{M}{(M-1)^2}.\] For $M>4$ , we have $\frac{|z|}{p(z)^2}< \frac{1}{2}$, and so \eqref{k1bd} holds. 
\end{proof}
 
Finally, we need only to show that $\left|\int k_2(y-x) t\varphi(t^2z)dz\right|$ is less than or equal to half of this lower bound on the $k_1$ term. 

\begin{clm}
Let $\varphi$ and $k_2$ be defined as above. Then
\[\left|\int k_2(y-x) t\varphi(t^2z)dz\right|\leq \frac{\sqrt\pi}{2}\int \frac{t|\varphi(t^2z)|}{\half{|y-z|}}dz\]
\end{clm}
\begin{proof}
Clearly, \[\left|\int k_2(y-x) t\varphi(t^2z)dz\right|\leq 4\int t|\varphi(t^2z)|dz.\] Now, all we need to show is that \[4\leq \frac{\sqrt\pi}{2\half{|y-z|}}.\] By the triangle inequality and our assumptions on $y$ and $M$ in the statement of the theorem, we have $|y-z|< |y|+|z|<\frac{\pi}{64}$. Then we have precisely that 
\[\frac{\sqrt\pi}{2\half{|y-z|}}\geq \frac{\sqrt\pi}{2}\frac{8}{\sqrt\pi}\geq 4,\] proving the claim. 
\end{proof}

By combining the two claims, we see that \begin{equation*}
 |u(t,yt^2)|\geq \frac{\sqrt\pi}{2}\int \frac{t|\varphi(t^2 z)|}{\half{|y-z|}}dz,
\end{equation*}
which completes the proof of the theorem. 
\end{proof}

The key ingredient in the proof above is the smallness of the support of $\varphi$. That assumption allows us to treat the kernel without worrying about the oscillatory factor $e^{\frac{i}{|y-z|}}$, which may contribute some cancellation in the region where $t^2 z\sim y$.  

\section*{Acknowledgements}
The author would like to thank the Institute for Mathematics and its Applications for its financial support during the theme year on Complex Fluids and Complex Flows, where this work was initiated and well as the NSF for its support through DMS grant 1101434. The author also 
thanks Sijue Wu for suggesting the problem and for many helpful discussions. 

\newpage
\appendix

\section{Precise Jacobian bounds}

In the following appendix, we collect the technical but not very deep results necessary to complete the proof of Theorem \ref{sharp}.

 Recall Proposition \ref{keysharp} relied on the lower bounds of certain Jacobian bounds. The Lemma below collects these bounds. 
\begin{lem}\label{Jacob}
\begin{enumerate}
\item 
If $J(\xi) =\dfrac{2\sgn(\xi-\xi_1) \sqrt{y\xi+|\xi|^a +(a-1)|\xi_1|^a)}}{y-a|\xi|^{a-1}}$ with $a>1$, then \[J(\xi) \geq \left\{\begin{array}{lr}
J(0) = \dfrac{2\sqrt{(a-1)|\xi_1|^a}}{y} = 2\half{(a-1)}a^{\frac{-a}{2(a-1)}}y^{\frac{a}{2(a-1)} -1}& a>2\\
&\\
J(-y^{\frac{1}{a-1}})=\dfrac{2 \sqrt{(a-1)|\xi_1|^a}}{(a-1)y} =\frac{2}{(a-1)^{\frac{1}{2}}}a^{\frac{-a}{2(a-1)}}y^{\frac{a}{2(a-1)}-1} &2>a>1\end{array}\right.\]\\
\item  
Let $\mathcal{J}(\xi) =\dfrac{-2\sgn(\xi-\xi_1) \sqrt{(1-a)|\xi_1|^a-y\xi+|\xi|^a}}{y|\xi|^{1-a}-a}$ with $0<a<1$; then,  \[\mathcal{J}(\xi) \geq \left\{\begin{array}{lr}
\mathcal{J}(0) = \dfrac{-2 \sqrt{(1-a)|\xi_1|^a}}{-a} = 2\half{(1-a)}a^{\frac{a-2}{2(a-1)}}y^{\frac{a}{2(a-1)}}& 1>a>1/2\\
&\\
\mathcal{J}(-y^{\frac{1}{a-1}})=\dfrac{2\sqrt{(1-a)|\xi_1|^a}}{{1-a}} =\frac{2}{(1-a)^{\frac{1}{2}}}a^{\frac{-a}{2(a-1)}}y^{\frac{a}{2(a-1)}} &1/2>a>0 \end{array}\right.\] 
\end{enumerate}
\end{lem}
\begin{proof}
First, observe that $J(\xi)$ is continuous. The only possible point of discontinuity is at $\xi_1$, but by l'Hopital's rule, 
\begin{eqnarray*}
\lim_{\xi\rightarrow \xi_1}J(\xi) &=& \lim_{\xi\rightarrow \xi_1} \dfrac{2\sgn(\xi-\xi_1) \sqrt{y\xi+|\xi|^a +(a-1)|\xi_1|^a}}{y-a|\xi|^{a-1}}\\
&=&\lim_{\xi\rightarrow \xi_1} \dfrac{\sgn(\xi-\xi_1) \left(y\xi+|\xi|^a +(a-1)|\xi_1|^a\right)^{-1/2}(y-a|\xi|^{a-1})}{a(a-1)|\xi|^{a-2}}\\
&=&\frac{2}{a(a-1)|\xi_1|^{a-2}}\lim_{\xi\rightarrow \xi_1}\frac{1}{J(\xi)}\\
\Rightarrow &&\lim_{\xi\rightarrow \xi_1}J(\xi) = \sqrt{2}(a(a-1)|\xi_1|^{a-2})^{-1/2}.
\end{eqnarray*} 
Since $J(\xi)$ is continuous, the natural way to find a lower bound is to consider the derivative of $J(\xi)$ and check for critical points. We will show that there are no critical points of $J(\xi)$ in the chosen interval, and therefore the lower bound is at one of the endpoints (which endpoint depends on the value of a). When $a=2$, all these machinations are unnecessary as $J(\xi) = C$. From this point forward, we will assume that $a\neq 2$. First, observe that the derivative of $J(\xi)$ is \[J'(\xi) = \frac{\sgn(\xi-\xi_1)\left[-2a(a-1)|\xi|^{a-2}(y\xi+|\xi|^a +(a-1)|\xi_1|^a)+(y-a|\xi|^{a-1})^2\right]}{(y-a|\xi|^{a-1})^2(y\xi+|\xi|^a +(a-1)|\xi_1|^a)^{1/2}}.\] Clearly, the numerator is $0$ at $\xi_1$. Since $\xi_1$ is also a zero of the denominator and it is easy to check using Taylor expansions that $J'(\xi_1)\neq 0$ and is, in fact, positive and finite (implying that $J'(\xi)$ is continuous), it suffices to check if the numerator $N(\xi)$ has any additional zeroes. Since \[N(\xi)=\sgn(\xi-\xi_1)\left[-
2a(a-1)|\xi|^{a-2}(y\xi+|\xi|^a +(a-1)|\xi_1|^a)+(y-a|\xi|^{a-1})^2\right]\] has a zero at $\xi_1$, the only way for $N$ to have additional zeroes is if $N'(\xi)$ is zero at a point besides $\xi_1$. Now, 
\[N'(\xi) = \sgn(\xi-\xi_1)2a(a-1)(a-2)|\xi|^{a-3}(y\xi +|\xi|^a +(a-1)|\xi_1|^a).\]   
By construction, $y\xi +|\xi|^a +(a-1)|\xi_1|^a\geq 0$ and equal to zero only at $\xi_1$, so the only additional possible zero is $0$ and then only when $a>3$. Therefore, $N(\xi)$ has no additional zeroes in the open interval $(-y^{\frac{1}{a-1}},0)$, and $J(\xi)$ is monotone increasing when $a>2$ and monotone decreasing when $1<a<2$ (since $N(\xi) \geq 0$ for $a>2$ and $N(\xi)\leq 0$ for $1<a<2$). Therefore, we have \[J(\xi) \geq \left\{\begin{array}{lr}
J(0) = \dfrac{2\sqrt{(a-1)|\xi_1|^a}}{y} = 2\half{(a-1)}a^{\frac{-a}{2(a-1)}}y^{\frac{a}{2(a-1)}}y^{-1}& a>2\\
&\\
J(-y^{\frac{1}{a-1}})=\dfrac{2 \sqrt{(a-1)|\xi_1|^a}}{(a-1)y} =2\mhalf{(a-1)}a^{\frac{-a}{2(a-1)}}y^{\frac{a}{2(a-1)}}y^{-1} &2>a>1
\end{array}\right.\] which completes the proof of part 1.  

Now, observe that $\mathcal{J}(\xi)$ is also continuous. The only possible point of discontinuity is at $\xi_1$, but by l'Hopital's rule, 
\begin{eqnarray*}
\lim_{\xi\rightarrow \xi_1}\mathcal{J}(\xi) &=& \lim_{\xi\rightarrow \xi_1} \dfrac{-2\sgn(\xi-\xi_1) \sqrt{(1-a)|\xi_1|^a-y\xi-|\xi|^a }}{{y|\xi|^{1-a}-a}}\\
&=&\lim_{\xi\rightarrow \xi_1} \dfrac{-\sgn(\xi-\xi_1) \left((1-a)|\xi_1|^a-y\xi-|\xi|^a \right)^{-1/2}(-y+a|\xi|^{a-1})}{-y(1-a)|\xi|^{-a}}\\
&=&\frac{1}{y(1-a)}\lim_{\xi\rightarrow \xi_1} \dfrac{|\xi|^a(y-a|\xi|^{a-1})}{-\sgn(\xi-\xi_1) \left((1-a)|\xi_1|^a-y\xi-|\xi|^a \right)^{1/2}}\\
&=&\frac{2|\xi_1|^{2a-1}}{a(1-a)|\xi_1|^{a-1}}\lim_{\xi\rightarrow \xi_1} \frac{1}{J(\xi)}\\
\Rightarrow &&\lim_{\xi\rightarrow \xi_1}J(\xi) = \sqrt{2|\xi_1|^{a}}(a(a-1))^{-1/2}.
\end{eqnarray*} 
Since $\mathcal{J}(\xi)$ is continuous, the natural way to find a lower bound is to consider the derivative of $\mathcal{J}(\xi)$, and (if it is continuous as well) check for critical points. We will show that $\mathcal{J}'(\xi)$ is strictly postive when $0<a<1/2$ and strictly negative when $1/2<a<1$.  When $a=1/2$, $\mathcal{J}(\xi) = C$, so we will assume $a\neq 1/2$. First, observe that the derivative of $\mathcal{J}(\xi)$ is \[\mathcal{J}'(\xi) = \frac{\sgn(\xi-\xi_1)\left[-2(1-a)y|\xi|^{1-a}((1-a)|\xi_1|^a-y\xi-|\xi|^a )+|\xi|^a|(y|\xi|^{1-a}-a)^2\right]}{|\xi|(y-a|\xi|^{a-1})^2(y\xi+|\xi|^a +(a-1)|\xi_1|^a)^{1/2}}.\] Clearly, the numerator is $0$ at $\xi_1$. Since $\xi_1$ is also a zero of the denominator and it is easy to check using Taylor expansions that $J'(\xi_1)\neq 0$ and is precisely $C(a)(1-2a)|\xi_1|^{5a/2 -3}$, where $C(a)>0$. The numerator also has a zero at $0$, but the $|\xi|$ will force $\mathcal{J}'(\xi)$ to go to positive or negative infinity as $\xi\nearrow 0$. In order to find 
critical points that could be extrema of $\mathcal{J}(\xi)$, it suffices to check if the numerator $N(\xi)$ has any additional zeroes. As in the case $a>1$, we will analyze the numerator with its derivative to check for zeroes. The numerator is slightly more complicated in this case, and it must have a critical point between $\xi_1$ and $0$ by Rolle's theorem. However, the derivative has other properties which will allow us to draw the necessary conclusions. 

\begin{clm}
The only zeroes of $N(\xi)$ are at $\xi=\xi_1,\,0$. Moreover, for $0<a<1/2$, $N(\xi)\leq 0$ for $\xi\in(-y^{1/(a-1)},\xi_1)$ and $N(\xi)\geq 0$ for $\xi\in(\xi_1, 0)$. For $1/2<a<1$, $N(\xi)\geq 0$ for $\xi\in(-y^{1/(a-1)},\xi_1)$ and $N(\xi)\leq 0$ for $\xi\in(\xi_1, 0)$  
\end{clm}
\begin{proof}
Rather than draw conclusions about arbitrary $a$, we will discuss $0<a<1/2$, but exactly the same arguments will yield similar conclusions for $1/2<a<1$, just with opposite signs.
Since $N(\xi)=-2(1-a)y|\xi|^{1-a}((1-a)|\xi_1|^a-y\xi-|\xi|^a)+|\xi|^a(y|\xi|^{1-a}-a)^2$, we can show that \begin{equation}\label{Na<1}|\xi|N'(\xi) +(1-a)N(\xi) = (1-2a)|\xi|^a(y|\xi|^{1-a}-a)^2.\end{equation} The right hand side is always the same sign except at its zeroes, $\xi_1$ and $0$. From this equation, we can conclude that $N'(\xi_1)=N''(\xi_1)=0$, but $N'''(\xi_1) = 2a^2(1-a)^2(1-2a)|\xi_1|^{a-3}$, so near $\xi_1$, the function $N(\xi)$ is a positive cubic. 
The equation (\ref{Na<1}) also implies that at any point $x\in(-y^{-1/(a-1)},0)$ not equal to $\xi_1$ or $0$ such that $N(x)=0$ must satisfy $N'(x)>0$. This fact means that in the subinterval $(\xi_1,0)$, there can be no additional zeroes of $N(\xi)$.  
Since $N(-y^{\frac{1}{a-1}})<0$ and $N(\xi)$ approaches zero from below as $\xi\nearrow\xi_1$, there can only be an even number of zeroes in the subinterval $(-y^{\frac{1}{a-1}},\xi_1)$. At one of the zeroes, $N(\xi)$ must be decreasing, but that would contradict equation (\ref{Na<1}). Therefore there are no zeroes of $N(\xi)$  in the subinterval $(-y^{\frac{1}{a-1}},0)$. Combining these two subintervals, we conclude that or $0<a<1/2$, $N(\xi)\leq 0$ for $\xi\in(-y^{1/(a-1)},\xi_1)$ and $N(\xi)\geq 0$ for $\xi\in(\xi_1, 0)$. Thus the claim is proved.    
\end{proof}
If we return to $\mathcal{J}'(\xi)$ and apply this claim, we find that when $0<a<1/2$, \, $\mathcal{J}'(\xi)>0$ for $\xi\in(-y^{\frac{1}{a-1}},0)$ and when $1/2<a<1$, $\mathcal{J}'(\xi)<0$ for $\xi\in(-y^{\frac{1}{a-1}},0)$. Thus, 
 \[\mathcal{J}(\xi) \geq \left\{\begin{array}{lr}
\mathcal{J}(0) = \dfrac{-2 \sqrt{(1-a)|\xi_1|^a}}{-a} = 2\half{(1-a)}a^{\frac{a-2}{2(a-1)}}y^{\frac{a}{2(a-1)}}& 1>a>1/2\\
&\\
\mathcal{J}(-y^{\frac{1}{a-1}})=\dfrac{2\sqrt{(1-a)|\xi_1|^a}}{{1-a}} =2\mhalf{(1-a)}a^{\frac{-a}{2(a-1)}}y^{\frac{a}{2(a-1)}} &1/2>a>0
\end{array}\right.\] 
\end{proof}

\newpage
\bibliography{papers2014.bib}
\bibliographystyle{plain}

\end{document}